\documentclass{article}
\usepackage{arXiv}
\usepackage{authblk}
\usepackage[sort&compress]{natbib}
\usepackage{graphicx}
\RequirePackage{fix-cm}
\RequirePackage{amsmath}
\RequirePackage{amssymb}
\RequirePackage{amsfonts}

\usepackage{graphicx}
\usepackage{pdfpages}
\usepackage{amsfonts,amsmath,amssymb}
\usepackage{mathrsfs}
\usepackage{mathptmx}
\usepackage{protosem}
\usepackage{caption}
\usepackage{epigraph}
\usepackage{enumitem}
\usepackage{mathtools}
\usepackage{bbm}
\usepackage{subcaption}

\usepackage{array}
\usepackage{multirow}
\usepackage{float}
\usepackage{booktabs}
\usepackage{bm}

\usepackage{todonotes}

\newcommand{\R}{\mathbbm{R}}
\newcommand{\C}{\mathbbm{C}}
\newcommand{\N}{\mathbbm{N}}

\newcommand{\Cheb}{\mathrm{Cheb}}
\newcommand{\Oc}{\mathcal{O}}
\newcommand{\ee}{\varepsilon}
\newcommand{\lo}{\longrightarrow}
\newcommand{\li}{\left}
\newcommand{\re}{\right}

\newcommand{\BLT}{\mathrm{BLT}}

\usepackage{todonotes}

\usepackage{amsthm}

{\bf}{\it}

\newtheorem{experiment}{\sc Experiment}

\begin{document}
\title{Multivariate Newton Interpolation in Downward Closed Spaces Reaches the Optimal Geometric Approximation Rates for Bos--Levenberg--Trefethen Functions}

\author{
Michael Hecht\thanks{Corresponding author. Email: m.hecht@hzdr.de}\textsuperscript{1,2,3},
Phil-Alexander Hofmann\textsuperscript{1,2},
Damar Wicaksono\textsuperscript{1,2},
Uwe Hernandez Acosta\textsuperscript{1,2},
Krzysztof Gonciarz\textsuperscript{6,7},
Jannik Kissinger,
Vladimir Sivkin\textsuperscript{4},
Ivo F. Sbalzarini\textsuperscript{5,6,7,8}
}

\affil{\textsuperscript{1}Center for Advanced Systems Understanding (CASUS), Görlitz, Germany}
\affil{\textsuperscript{2}Helmholtz-Zentrum Dresden-Rossendorf e.V. (HZDR)}
\affil{\textsuperscript{3}Mathematical Institute, University of Wrocław}
\affil{\textsuperscript{4}HSE University, Moscow, Russia}
\affil{\textsuperscript{5}Technische Universität Dresden, Faculty of Computer Science, Dresden, Germany}
\affil{\textsuperscript{6}Max Planck Institute of Molecular Cell Biology and Genetics, Dresden, Germany}
\affil{\textsuperscript{7}Center for Systems Biology Dresden, Dresden, Germany}
\affil{\textsuperscript{8}Center for Scalable Data Analytics and Artificial Intelligence (ScaDS.AI), Dresden, Germany}


\maketitle

\begin{abstract}
{We extend the univariate Newton interpolation algorithm to arbitrary spatial dimensions and for any choice of downward-closed polynomial space, while preserving its quadratic runtime and linear storage cost. The generalisation supports any choice of the provided notion of  non-tensorial unisolvent interpolation nodes, whose number coincides with the dimension of the chosen-downward closed space. Specifically, we prove that by selecting Leja-ordered Chebyshev-Lobatto or Leja nodes,
the optimal geometric approximation rates for a class of analytic functions---termed Bos--Levenberg--Trefethen functions---are achieved and extend to the derivatives of the interpolants.
In particular, choosing Euclidean degree results in downward-closed spaces whose dimension only grows sub-exponentially with spatial dimension, while delivering approximation rates close to, or even matching those of the tensorial maximum-degree case, mitigating the curse of dimensionality.
Several numerical experiments demonstrate the performance of the resulting multivariate Newton interpolation compared to state-of-the-art alternatives and validate our theoretical results.}
{Newton interpolation,  unisolvent nodes, analytic functions, geometric approximation, non-tensorial grids.}
\end{abstract}

\maketitle


\section{Introduction}
\label{intro}
Polynomial interpolation goes back to Newton, Lagrange, and others  \citep[see, e.g.,][]{LIP}, and its fundamental importance in mathematics and computing is undisputed.
Interpolation is based on the fact that, in 1D, one and only one polynomial $Q_{f,n}$ of degree $n$ can interpolate a function
$f : \R \lo \R$ in $n+1$ distinct \emph{unisolvent interpolation nodes} $P_n \subseteq \R$, $Q_{f,n}(p_i) =f(p_i)$ for all $p_i \in P_n$, $0\leq i \leq n$.
Though the famous \emph{Weierstrass approximation theorem} \citep{weier1}
states that any continuous function $f \in C^0(\square_m)$, $\square_m=[-1,1]^m$, $\|f\|_{C^0(\square_m)} = \sup_{x \in \square_m}|f(x)|<\infty$ can be uniformly approximated by
polynomials, this does not necessarily apply for interpolation. In contrast to interpolation, the Weierstrass approximation theorem does not require the polynomials to coincide with $f$ anywhere, meaning there is a sequence of polynomials $Q_{f,n}$ with $Q_{f,n}(x) \not = f(x)$ for all $x\in \square_m$, but still
\begin{equation}
Q_{f,n} \xrightarrow[n \rightarrow \infty]{} f \quad \text{uniformly on}\,\,\, \square_m\,.
\end{equation}
There are several constructive proofs of the Weierstrass approximation theorem, including the prominent version given by Serge~\citet{bernstein1912}. Although the resulting Bernstein approximation scheme
is \emph{universal} (i.e., approximating any continuous function) and has been proven to reach the optimal (inverse-linear) approximation rate for the absolute value function $f(x) = |x|$ \citep{bernstein1914},
it achieves only slow convergence rates for analytic functions, resulting in a high computational cost in practice.

In contrast, interpolation in \emph{Chebyshev, Legendre, Pad\'e, or Leja nodes} \citep{Bos2010,trefethen2019} is known to be \emph{non-universal}~\citep{faber},
but ensures the approximation of Lipschitz continuous functions---Runge's overfitting phenomenon completely disappears---with exponential
approximation rates appearing for analytic functions~\citep{trefethen2019,chifka2013}.

There has thus been much research into multi-dimensional ($m$D) extensions of one-dimensional (1D) interpolation schemes and their approximation capabilities. While multivariate \(C^k\) smooth functions can be approximated at a maximal algebraic rate of \(\Oc(n^{-k/m})\) \citep{Devore1989, novak2}, we extend the discussion based on the results of \cite{bos2018bernstein,trefethen2017a}, addressing the question of which multivariate function class can be approximated by polynomials with a geometric rate.

\subsection{Bos--Levenberg--Trefethen functions}

Consider the multi-index sets
$A_{m,n,p} = \li\{\alpha \in \N^m :  \|\alpha\|_p \leq n \re\} \subseteq \N^m$ of bounded $l_p$-norm and the induced polynomial spaces
$\Pi_{m,n,p} = \mathrm{span}\{x^\alpha = x^{\alpha_1}\cdots x^{\alpha_m}\}_{ \alpha \in A_{m,n,p}}$, generalising the notion of polynomial degree
to multi-dimensional $l_p$-degree, with \emph{total degree}, \emph{Euclidean degree}, and \emph{maximum degree} appearing for the choice of $p=1,2,\infty$, respectively.

Assume that a given continuous function $f : \square_m \lo \R$ on the hypercube $\square_m = [-1,1]^m$, possesses
a Chebyshev series expansion (holding true for any Lipschitz continuous function \citep[Theorem~4.1]{mason1980})
\begin{equation}\label{eq:cheb}
 f(x) = \sum_{\alpha \in \N^m}c_\alpha T_\alpha(x)\,,
 \quad c_\alpha = \li<\omega f,T_{\alpha}\re>_{L^2(\square_m)} = \int_{\square_m} \omega(x)f(x) T_\alpha(x) dx\,,
\end{equation}
where $T_\alpha(x) = \prod_{i=1}^m T_{\alpha_i}(x_i)$ is the product of the univariate Chebyshev polynomials of order $\alpha_i$, the coefficients $c_\alpha$ are given by the orthogonal, $\omega$-weighted $L^2$-projection, where $\omega(x) = 2^{m-a}/\pi^m\prod_{i=1}^m (1-x_i^2)^{-1/2}$, with $a$ being the number of zero entries of $\alpha$  \citep{trefethen2019}. Then, the truncation of the Chebyshev series to $\Pi_{m,n,p}$ can reach the following approximation rates:

\begin{theorem}[\cite{trefethen2017a}]
Given a continuous function $f : \square_m \lo \R$ satisfying Eq.~\eqref{eq:cheb}, assume that $f$ possesses an analytic (holomorphic) extension to the \emph{Trefethen domain}
\begin{equation}\label{Tdomain}
   N_{m,\rho} = \li\{ (z_1,\dots,z_m) \in \C^m :  (z_1^2 + \cdots  + z_m^2) \in E_{m,h^2}^2 \re\}\,, \quad m \in \N\,,
 \end{equation}
where
$E_{m,h^2}^2$ denotes the \emph{Newton ellipse} with foci $0$ and $m$ and leftmost point $-h^2$, $h \in [0,1]$.
Setting $\rho = h + \sqrt{1 + h^2}$,
the following \emph{upper bounds} on the convergence rate of the truncation
$\mathcal{T}_{A_{m,n,p}}(f)  =\sum_{\alpha \in A_{m,n,p}}c_\alpha T_{\alpha}\in \Pi_{m,n,p}$
apply:
\begin{equation}\label{eq:rate}
  \| f - \mathcal{T}_{A_{m,n,p}}(f)\|_{C^0(\square_m)} =
  \begin{cases}
    \mathcal{O}_\epsilon(\rho^{-n/\sqrt{m}}) & ,\, p = 1 \\
    \mathcal{O}_\epsilon(\rho^{-n}) & ,\, p = 2 \\
    \mathcal{O}_\epsilon(\rho^{-n}) & ,\, p = \infty \, ,
  \end{cases}
\end{equation}

where  $g \in \Oc_\ee(\rho^{-n})$ if and only if $g \in \Oc((\rho-\ee)^{-n})$,  $\forall \rho > \ee >0$.
\label{theo:Nick}
\end{theorem}
Note that the number of coefficients for total degree interpolation $|A_{m,n,1}| = \binom{m+n}{n} \in \Oc(m^n) \cap \Oc(n^m)$ scales polynomially,
for Euclidean degree $|A_{m,n,2}|\approx \frac{(n+1)^m }{\sqrt{\pi m}} \li(\frac{\pi \mathrm{e}}{2m}\re)^{m/2} \in o(n^m)$ scales sub-exponentially, whereas
for maximum degree $|A_{m,n,\infty}| = (n+1)^m$ scales exponentially with the dimension $m \in \N$. Consequently, in case the exponential rate, Eq.~\eqref{eq:rate}, applies, approximating  functions with respect to Euclidean degree might resist the curse of dimensionality, while approximation with total or maximum degree results to be sub-optimal.

This observation motivated \citet{trefethen2017a} to conjecture the converse statement to hold:
If a function $f : \square_m \lo \R$ possesses a polynomial approximation of exponential approximation rate $\Oc(\rho^{-n})$, then it can be analytically extended to $N_{m,\rho}$. By generalising Bernstein--Walsh theory to functions $f: K \lo \C$ defined on \emph{PL-regular, compact domains} $K \subseteq \C^m$ (including the case $K = \square_m$) \citet{bos2018bernstein} extended Trefethen's statement, in particular proving a refined version of the conjecture:


\begin{theorem}[\cite{bos2018bernstein}]\label{theo:BL} Let $K \subseteq \C^m$, $m \in \N$, be compact and PL-regular, $f: K \lo \C$ be continuous.
  Denote by $\Pi(nP)$ the polynomial space induced by a \emph{convex body} $P \subseteq \R^{m,+}$
  (including the cases $\Pi_{m,n,p}$). Let $\rho>1$ and $\Omega_{\rho(P,K)}:= \{z \in \C^m : V_{P,K}(z) < \log(\rho)\}$, where $ V_{P,K}(z) = \lim_{n\rightarrow\infty} \sup_{p\in \Pi(nP)}\{\frac{1}{n}\log|p(z)| :  \|p\|_{C^0(K)} \leq 1\}$.


  \begin{enumerate}[label=\roman*)]
    \item If $f$ is the restriction to $K$ of a function holomorphic in $\Omega_\rho(P,K)$, then
\begin{equation}\label{eq:bestR}
   \|f - p_n^*\|_{C^0(K)}  \lesssim \rho^{-n}\,,
\end{equation}
where $p_n^* \in \Pi(nP)$, with $\|f - p_n^*\|_{C^0(K)}=\inf_{p_n \in \Pi(nP)} \|f - p_n\|_{C^0(K)}$, denotes the \emph{best approximation} of $f$ in $\Pi(nP)$.
    \item If $  \|f - p_n^*\|_{C^0(K)}  \lesssim \rho^{-n}$, then $f$ is the restriction to $K$ of a function holomorphic in $\Omega_{\rho(P,K)}$.
  \end{enumerate}
  \label{theo:bos}
\end{theorem}
While in the hypercube $K=\square_m$ we show the \emph{best approximation} of exponential rate to induce geometric \emph{near-best interpolation}, Theorem~\ref{theo:APP},
Theorems~\ref{theo:Nick}~and~\ref{theo:bos} motivate us to define the following function class:


\begin{definition}[Bos--Levenberg--Trefethen functions] Let $K \subseteq \C^m$, $m \in \N$, be compact and PL-regular, we call the class of functions $\BLT(K) \subseteq C^0(K,\C)$ that are
restrictions of functions being holomorphic in $\Omega_\rho(P,K)$, $\rho >1$,
    \emph{Bos-Levenberg-Trefethen (BLT)-functions}.
\end{definition}
\begin{remark}\label{rem:BLT}
In contrast to the previously introduced unbounded Trefethen domain $N_{m,\rho}$, BLT-functions only need to be holomorphic in the bounded pre-compact
\emph{Bos--Levenberg domain} $\Omega_\rho(P,K) \subset\subset \C^m$. Containing Trefethen's former notion and all entire functions, the BLT functions are a large function class, covering many approximation tasks that frequently arise in applications. However, this comes at the cost that Euclidean-degree approximation will not always deliver the same rate as maximum-degree approximation, as it holds for functions $f$ being restrictions of a function holomorphic in $N_{m,\rho}$.
\end{remark}

We illustrate Remark~\ref{rem:BLT}:
Throughout this article, we focus on the case $K = \square_m$, for which the famous
\emph{Runge function}
\begin{equation}\label{eq:Runge}
f : \square_m \lo \R\,, \quad f(x) = \frac{1}{s^2 + r^2\|x\|^2}\,, \quad r,s  \neq 0\,,
\end{equation}
is a BLT function.
As a consequence of Theorem~\ref{theo:bos}, \cite{bos2018bernstein} explicitly proved the approximation rates in Eq.~\eqref{eq:rate} to apply, $ \|f - p_n^*\|_{C^0(K)}  \lesssim \rho^{-n}$, with
\begin{equation}\label{eq:Rbounds}
\rho= \begin{cases}
\frac{h + \sqrt{h^2 + m}}{\sqrt{m}} & \text{if } p = 1 \\
h + \sqrt{h^2 + 1} & \text{if } 2 \leq p \leq \infty
\end{cases}\,, \quad  h = \frac{s}{r}\,,
\end{equation}
indeed showing the choice of Euclidean degree to be optimal.

However, as aforementioned, the choice of Euclidean degree is not optimal for all BLT functions. \cite{bos2018bernstein} computed $\rho_2 = 2.0518 < \rho_\infty = 2.1531$ for the $l_p$-degree choices $p=2,\infty$, respectively, in the case of the shifted Runge function
\begin{equation}\label{eq:BLT}
f : \square_2 \lo \R \,, \quad   f(x,y) = \frac{1}{(x-a)^2 + (y-a)^2} \,, \quad a =5/4\,.
\end{equation}

In light of these facts, two questions arise:

\begin{enumerate}[left=0pt,label=\textbf{Q\arabic*)}]
\item\label{Q1} How to stably and efficiently compute near-best polynomial approximations $f \approx Q_f \in \Pi(nP)$, by sampling $f$ at only
$\dim \Pi(nP)$-many nodes?

Though \citet{trefethen2017a} demonstrated the optimal Euclidean approximation rate to apply for the Runge function (in the 2D case $m=2$), this was realized by least-square regression on a fine grid, not answering this question.

\item \label{Q2}Given a BLT function $f : \square_m \lo \R$, how to identify the polynomial space $\Pi(nP)$ such that the relative rate
\begin{equation}
  \|f-Q_{f,n}\|_{C^0(\square_m)}\frac{ \dim \Pi(nP)}{\dim \Pi_{m,n,\infty}}
\end{equation}
emerges as optimal among the potential choices? This question was already raised by \cite{cohen3}.
\end{enumerate}

We next detail our contribution, addressing these questions, in relation to former approaches.

\subsection{Related work and contribution}


While \emph{tensorial Chebyshev interpolation} is a well-established interpolation scheme, as for example realised in the prominent MATLAB package {\sc chebfun} \citep{chebfun},
it is limited to the maximum-degree case and, so far, only implemented up to dimension $m=3$, reflecting its non-resilience to the curse of dimensionality. Sparse tensorial interpolation as proposed by \cite{kuntz,Guenther,sauertens,dyn},
delivers high-dimensional function approximations efficiently. However, it does not apply for the present general definition of the spaces $\Pi(nP)$.

If $\Pi(nP) = \Pi_{A}=\mathrm{span}\{x^\alpha : \alpha \in A\}$, with $A$ being \emph{downward closed}, interpolation in \emph{Leja points} or more general \emph{nested node sets} has been proposed by \cite{cohen2,cohen3}.
Further studies of the resulting Lebesgue constants and approximation power were provided by \cite{chifka2013,Narayan2014,Nobile2014,griebel2016tensor}.

However, the underlying interpolation algorithms require super-quadratic $\Omega(|A|^2)$ up to cubic runtime $\mathcal{O}(|A|^3)$. The resulting high computational cost might be dominated by the high sampling costs of the function $f: \square_m \to \R$, as presumed in the case of parametric PDEs by \cite{cohen2}. As a result, applications are hampered by the runtime and storage requirements, limiting the addressable dimension and instance sizes. This might be the reason why, apart from the maximum-degree case, none of those approaches has yet been demonstrated to reach the optimal approximation rates for BLT functions, especially not in dimensions $m \geq 4$.

Our contribution focuses both on resolving the algorithmic issues and on achieving optimal approximation power:

\begin{enumerate}[left=0pt,label=\textbf{C\arabic*)}]
\item \label{CC1} By extending our previous work \citep{PIP1,PIP2,MIP}, we contribute to solving the interpolation task for arbitrary downward closed polynomial spaces $\Pi_A=\mathrm{span}\{x^\alpha : \alpha \in A\}$, $A\subseteq \N^m$, by delivering a \emph{multivariate (Newton) interpolation algorithm (MIP)} of \emph{quadratic runtime}, $\mathcal{O}(|A|^2)$, and \emph{linear storage}, $\mathcal{O}(|A|)$, Theorem~\ref{theo:MIP}.

Hereby, MIP samples the function $f: \square_m \to \R$ solely in \emph{unisolvent non-tensorial interpolation nodes} $P_A\subseteq \square_m$ of size $|P_A| = \dim P_A = |A|$, answering \ref{Q1}. In particular, MIP is highly flexible in choosing a particular set of unisolvent nodes, relaxing former stiffer implementations such as {\sc chebfun} \citep{chebfun}.

\item \label{CC2} In Theorem~\ref{theo:APP}, we prove that the geometric rate of the best approximation of a BLT-function and its derivatives extends to the interpolant in suitable grids, Lemma~\ref{lemma:LEB}, such as \emph{Leja point grids (LP nodes)} $$\| f - Q_{f,n} \|_{C^k(\square_m)} = \mathcal{O}_{\ee}(\rho^{-n})\,, \quad k \in \N\,.$$

\item \label{CC3} While \emph{barycentric Lagrange interpolation} \citep{berrut,trefethen2019} is known as a pivotal choice in 1D, enabling numerically stable interpolation up to degree $n \approx 1.000.000$,
in $m$D only degrees up to $n \approx 1000$ might be computable.

That is why MIP extends 1D Newton interpolation to $m$D, which is known to be stable in this range of degrees for Leja-ordered nodes \citep{tal1988high}.

Apart from its numerical stability, we empirically demonstrate that, when choosing LP nodes or \emph{Leja-ordered Chebyshev--Lobatto nodes (LCL nodes)}, MIP reaches the optimal approximation rates for several BLT functions. Hereby, the Euclidean degree ($p = 2$) emerges as pivotal choice for mitigating the curse of dimensionality, which provides at least an empirical answer to \ref{Q2}.

Moreover, we prove that posterior evaluation and $k$-th order differentiation of MIP-interpolants can be realised efficiently in $\mathcal{O}(m|A|)$ and $\mathcal{O}(mn^k|A|)$, respectively (Theorem~\ref{theo:DIFF}) and we numerically demonstrate the maintenance of the optimal geometric rates for up to $2$-nd order derivatives.

\end{enumerate}

\subsection{Notation}\label{sec:NOT}
\begin{table}[!hpt]
\centering
\renewcommand{\arraystretch}{1.15}
\setlength{\tabcolsep}{4pt}
\small
\begin{tabular}{|c|l||c|l||c|l|}
\hline
$\square_m$ & $m$-dimensional hypercube & $\Pi_A$, $\Pi_{m,n,p}$ & polynomial space & $n$ & polynomial degree \\
$A$, $A_{m,n,p}$ &  multi-index set &
$\alpha, \beta \in A $ & multi-indices  & $i,j,k$ & indices, integers \\  $|\cdot|$ &  cardinality &
 $\|\cdot\|_p$ & \(\ell^p\)-norm  & $\text{span}$ & linear hull \\
  $C^{k}(\square_m)$ &  space of  differentiable functions  & $\quad\|\cdot\|_{C^k(\square_m)}$ & $C^k$-norm  &
 $e_i$ & standard basis \\
 $P_A$ & unisolvent nodes & $\Lambda$ & Lebesgue constant & $\lesssim$ & asymptotically smaller
 \\
\hline
\end{tabular}
\caption{Notation used throughout the article.}
\label{tab:notation}
\end{table}

Let $m,n \in \N$, $p>0$. Throughout this article, $\square_m=[-1,1]^m$ denotes the $m$-dimensional \emph{standard hypercube}.
We denote by  $A_{m,n,p} \subseteq \N^m$ all multi-indices $\alpha =(\alpha_1,\dots,\alpha_m)\in \N^m$ with $l_p$-norm $\|\alpha\|_p  \leq n$, $1\leq p \leq \infty$.
We order a finite set $A\subseteq \N^m$, $m \in \N$, of multi-indices with respect to the lexicographical order $\leq_L$ on $\N^m$ proceeding from the last entry to the first, e.g.,
$(5,3,1)\leq_L (1,0,3) \leq_L (1,1,3)$.
A multi-index set $A \subseteq \N^m$ is called
\emph{downward closed} (also termed \emph{monotone} or \emph{lower set}) if and only if
$\alpha = (a_1,\dots,a_m) \in A$ implies $\beta = (b_1,\dots,b_m) \in A$
whenever $b_i \leq a_i$, $ \forall \,  i=1,\dots,m$.
The sets $A_{m,n,p}$ are downward closed for all $m,n\in \N$, $p>0$ and induce the generalised notion of \emph{polynomial $l_p$-degree}.


We denote by $\Pi_m$ the $\R$-\emph{vector space of all real polynomials} in $m$ variables.
For $A\subseteq \N^m$, $\Pi_{A} \subseteq \Pi_m$ denotes the \emph{polynomial subspace} $\Pi_A = \mathrm{span}\{x^\alpha\}_{\alpha \in A}$ spanned by the \emph{canonical} \emph{basis}, whereas \emph{total degree} $A = A_{m,n,1}$, \emph{Euclidean degree} $A= A_{m,n,2}$, and \emph{maximum degree}
$A= A_{m,n,\infty}$ are of particular interest. We abbreviate $\Pi_{m,n,p} = \Pi_{A_{m,n,p}}$.

With $C^0(\square_m)$ we denote the
\emph{Banach space}
of continuous functions $f : \square_m \lo \R$ with norm
\linebreak
$\|f\|_{C^0(\square_m)} = \sup_{x \in \square_m}|f(x)|$ and with $C^k(\square_m)$, $k \in \N$,
$\|f\|_{C^k(\square_m)} =  \sum_{\beta \in A_{m,n,1}}\|\partial_\beta f\|_{C^0(\square_m)}$, $\partial_\beta f (x) = \partial^{l}_{x_1^{\beta_1}\cdots x_m^{\beta_m}}f(x)$, $\|\beta\|_1=l \leq k$, the Banach space of functions continuously differentiable in the interior of $\square_m$ up to $k$-th order.

Further notation is summarised in Table \ref{tab:notation}.

%


\section{The notion of unisolvence}
\label{sec:UN}

Essential for polynomial interpolation is the uniqueness of the interpolant $Q_{f,A} \in \Pi_A$, $Q_{f,A} (p_\alpha) = f(p_\alpha)$, $ \forall \alpha \in A \subseteq \N^m$, of a function
$f : \R^m \lo \R$. Interploation nodes $P_A\subseteq R^m$ guaranteeing the uniqueness are called \emph{unisolvent nodes} with respect to $\Pi_A$. Equivalently, unisolvent nodes $P_A$ exclude the existence of a non-zero polynomial $Q \in \Pi_A\setminus \{0\}$ vanishing on $P_A$, $Q(p_\alpha) =0$, $\forall \alpha \in A$.

The pioneering works of \citet{kuntz} and \citet{Guenther} with extensions by \citet{Chung} proposed constructions of unisolvent nodes $P_A \subseteq \square_m$
for the cases $A= A_{m,n,1}, A_{m,n,\infty}$. An explicit extension to the case of arbitrary downward closed spaces $\Pi_A$ has been given by \citet{cohen2,cohen3}.
Here, we provide a more general construction leading directly to a notion of unisolvence that permits implementing the initially announced MIP-algorithm.

\subsection{Unisolvent nodes}
We provide a constructive notion of unisolvence, resting on the follwoing defintions:

\begin{definition}[Transformations] An \emph{affine transformation}  $\tau : \R^m \lo \R^m$, $m \in \N$,
is a map $\tau(x) = Bx +b$, where $B \in \R^{m\times m}$ is an invertible matrix and $b \in \R^m$. An \emph{affine translation} is an affine transformation with $B=I$ the identity matrix. A
\emph{linear transformation} is an affine transformation with $b=0$.
\end{definition}
In this definition, the following holds:
  \begin{lemma}\label{lemma:TRF} Any affine transformation $\tau : \R^m \lo \R^m$, $m \in \N$, induces a ring isomorphism \linebreak $\tau^*: \R[x_1,\dots,x_m]\lo \R[x_1,\dots,x_m]$,
   $ \tau^*(Q)(x) = Q(\tau(x)) \,, \forall \, x \in \R^m$.
That is:
\begin{enumerate}[label=\roman*)]
 \item $\tau^*(1)=1$,
 \item  $\tau^*(\lambda Q_1 + \mu Q_2) = \lambda\tau^*(Q_1) + \mu \tau^*(Q_2)$ for all $Q_1,Q_2 \in \Pi_m$ and $\lambda,\mu \in \R$,
 \item  $\tau^*(Q_1Q_2) = \tau^*(Q_1)\tau^*(Q_2)$ for all $Q_1,Q_2 \in \Pi_m$ ,
 \item  the extension of $\tau^*$ to the ring of rational functions $\R[x_1,\dots,x_m]$ fulfills
 $\tau^*(Q_1/Q_2) =$  \linebreak $\tau^*(Q_1)/\tau^*(Q_2)$ for all $Q_1,Q_2 \in \Pi_m$, $Q_2 \neq 0$.
\end{enumerate}
\end{lemma}

\begin{proof} While $(i)$ is trivial and $(ii)$, $(iii)$ are straightforward to prove, $(iv)$ follows from $(iii)$ using the identity
 $\tau^*(Q_1)=\tau^*(1\cdot Q_1) = \tau^*((Q_2/Q_2)Q_1) = \tau^*(Q_2)\tau^*(Q_1/Q_2)$.
\end{proof}

\begin{definition}If $\Pi \subseteq \Pi_m$, $m \in \N$, is a finite-dimensional polynomial subspace, then we call $\tau : \R^m \lo \R^m$ a \emph{canonical transformation} with respect to $\Pi$ if
and only if $\tau$ is an affine transformation such that the induced transformation
$\tau^* : \Pi \lo \Pi_m$ satisfies  $\tau^*(\Pi)\subseteq \Pi$, resulting in $\tau^*$ to be an automorphism of $\Pi$.
\end{definition}

Note that not all transformations $\tau$ are canonical:
\begin{example}
Consider $Q(x,y) = x^ky^k$, $k \in \N$ and $\tau(x,y) = (x+y,x-y)$. Then $Q(\tau(x,y)) = (x+y)^k(x-y)^k = x^{2k} + \ldots $.
When choosing $k=\lfloor n /\sqrt{2}\rfloor$ maximal, such that $Q \in \Pi_{2,n,2}$, we deduce  $2k =2 \lfloor n /\sqrt{2}\rfloor >  n $ for $n \gg 1$. Hence, $\tau^*(Q) \not \in \Pi_{2,n,2}$, implying that $\tau$ is not canonical with respect to $\Pi_{2,n,2}$.
\end{example}
We continue formalising the concept of unisolvent nodes:
\begin{definition}[Unisolvence for hyperplane splits] Let $m \in \N$ and $\Pi \subseteq \Pi_m$ be a finite-dimensional polynomial subspace,
and let $H\subseteq \R^m$, $H = Q_H^{-1}(0)$ be a hyperplane defined by a linear polynomial $Q_H\in \Pi_{m,1,1}\setminus \{0\}$,
such that any affine transformation $\tau_H : \R^m \lo \R^m$ with $\tau_H(H) = \R^{m-1} \times \{0\}$ is canonical with respect to $\Pi$.
We consider
\begin{align}
 \Pi_{|H} &= \li \{ Q \in \Pi : \tau_H^*(Q) \in \Pi \cap (\Pi_{m-1}\times\{0\})\re\} \label{PIH} \\
 \Pi_{|H}^\# &= \li \{ Q \in \Pi_m : Q_H Q \in \Pi\re\} \nonumber
\end{align}
and call $P \subseteq \R^m$ \emph{unisolvent with respect to the hyperplane splitting} $(\Pi, H)$ if and only if:
\begin{enumerate}[label=\roman*)]
 \item there is no polynomial $Q \in \Pi_{|H}$ with $\tau^*_H(Q) \not = 0$ and $Q(P\cap H) =0$, and
 \item there is no polynomial $Q \in \Pi_{| H}^\# \setminus \{0\}$ with $Q(P\setminus H) =0$.
\end{enumerate}
\end{definition}

\begin{figure}[t!]
\center
\includegraphics[scale=0.15]{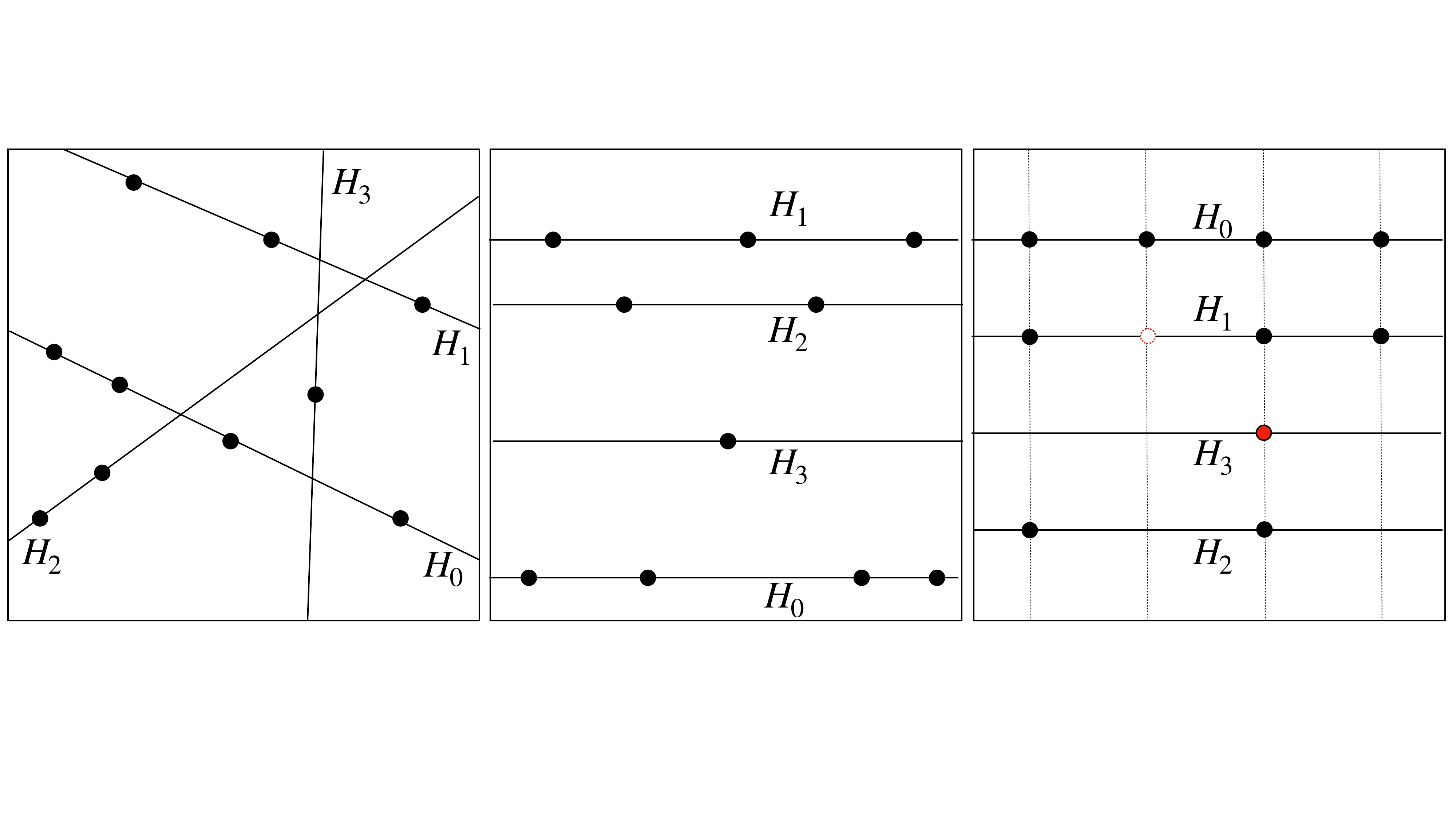}
\caption{Examples of unisolvent nodes $P_A$ for $A= A_{2,3,1}$ in general (left), irregular (middle), and non-tensorial (right) grids. In the right panel, non-tensorial nodes are indicated in red with missing symmetric counterparts shown as open symbols. \label{Fig:UN}}
\end{figure}

With the provided ingredients we state:
\begin{theorem} Let $m \in \N$, $ \Pi \subseteq \Pi_m$ be a finite-dimensional polynomial subspace, $P \subseteq \R^m$ a finite set of nodes, and $H = Q_H^{-1}(0)$ be a hyperplane of co-dimension 1 defined by a polynomial $Q_H \in \Pi_{m,1,1}\setminus\{0\}$
such that:
\begin{enumerate}[label=\roman*)]
 \item\label{A1}  the affine transformation $\tau_H : \R^m \lo \R^m$ with $\tau_H(H) = \R^{m-1} \times \{0\}$  induces a canonical transformation $\tau^*_H : \Pi \lo \Pi$, and
 \item\label{A2}  $P$ is unisolvent with respect to the hyperplane splitting $(\Pi,H)$.
\end{enumerate}
Then $P$ is unisolvent with respect to $\Pi$.
\label{theorem:UN}
\end{theorem}

\begin{proof} Let $Q \in \Pi$ with $Q(P)=0$.
We consider the affine transformation $\tau_H : \R^m \lo \R^m$ with $\tau_H(H) = \R^{m-1} \times \{0\}$ and the projection $\pi_{m-1}: \Pi_m \lo \Pi_{m-1} \times \{0\}$.
Let further
\begin{equation}\label{split}
 Q_1 = \tau^{*-1}_H\pi_{m-1}\tau^*_H(Q) \in \Pi_H \quad \text{and}\quad  Q_2= (Q -Q_1)/Q_H\,.
\end{equation}

{\bf Step 1:} We show that $Q_2 \in \Pi_{H}^\# $.  Certainly, $Q_2$ is a well-defined function on $\R^m\setminus H$.
Furthermore, we note that the linearity of $\tau_H$ implies  $\tau_H^*(Q_H) = \lambda x_m$, $\lambda \in \R\setminus \{0\}$. W.l.o.g.,~we assume $\lambda=1$ and use Lemma \ref{lemma:TRF}$iii)$ to reformulate Eq.~\eqref{split} as
\begin{align*}
 Q_2 &=  \tau^{*-1}_H\big(\tau^{*}_H(Q) - \pi_{m-1}\tau^*_H(Q)\big)\big /(\tau^{*-1}_H (\tau^{*}_H(Q_H))   \\
 &= \tau^{*-1}_H\big((\tau^{*}_H(Q) - \pi_{m-1}\tau^*_H(Q))/x_m\big) .
\end{align*}
Since $Q_0:=\tau_H^*(Q)- \pi_{m-1}\tau^*_H(Q)$ can be expanded into canonical form, of which all monomials share the variable $x_m$, the quotient $(\tau^{*}_H(Q) - \pi_{m-1}\tau^*_H(Q))/x_m \in \Pi$ is a polynomial, implying $Q_2 \in \Pi$.
Further, by Lemma \ref{lemma:TRF}$ii)$, we obtain
$$Q_HQ_2 = \tau_H^{*-1}(x_m)\tau_H^{*-1}(Q_0 /x_m)=\tau_H^{*-1}(Q_0) \in \Pi.$$
Hence, $Q_2 \in \Pi_{| H}^\# $ as claimed.

{\bf Step 2:} We show that $Q=0$. Indeed, $Q(p) = Q_1(p) =0$ for all $p \in P\cap H$ implies that $Q_1=0$ due to assumption~\ref{A1}. Consequently, $Q_HQ_2(p) =0$ for all $p\in P\setminus H$.
Since $Q_H(p)\not =0$ for all $p \in P\setminus H$, we get $Q_2(p) =0$, $\forall p \in P\setminus H$. Since $P$ is unisolvent with respect to
the hyperplane splitting $(\Pi,H)$, and due to Step 1, we have $Q_2 \in \Pi_{| H}^\# $, this implies $Q_2=0$.
Thus, $Q=0$ is the zero polynomial, proving $P$ to be unisolvent with respect to $\Pi$.
\end{proof}

\begin{example}In Fig.~\ref{Fig:UN}, we show examples of unisolvent nodes in 2D for $A= A_{2,3,1}$, generated by recursively applying Theorem~\ref{theorem:UN}. The three panels show examples for three different choices of the, in this case 1D, hyperplanes $H_0,\ldots ,H_3$ (solid lines) from Theorem~\ref{theorem:UN}. In the left panel, the hyperplanes are chosen arbitrarily. This starts by first choosing a hyperplane (line) $H_0$ and $n+1=4$ unisolvent nodes on $H_0$. Then choose $H_1 \not = H_0$ and 3 unisolvent nodes on $H_1 \setminus H_0$,
and recursively continue until choosing 1 unisolvent node on $H_3\setminus (H_0 \cup H_1 \cup H_2)$. When choosing the hyperplanes parallel to each other, as shown in the middle panel, this construction results in an irregular grid. Quantizing the distance between hyperplanes as well as between nodes on them further leads to non-tensorial grids, as shown in the right panel.
\end{example}

This example illustrates how the notion of unisolvence presented here extends beyond notions resting on (sparse) symmetric, tensorial, or nested grids, such as Leja points \citep{cohen2,cohen3}. However, even this generalised notion admits multivariate interpolation algorithms, thanks to the following splitting statement:
\begin{theorem}\label{GS} Let the assumptions of Theorem \ref{theorem:UN} be fulfilled and $f : \R^m \lo \R$ be a function.
Assume there are polynomials $Q_1 \in \Pi_{|H}$, $Q_2 \in \Pi_{| H}^\# $ with $\Pi_{|H}$, $\Pi_{| H}^\# $ from Eq.~\eqref{PIH}, such that:
\begin{enumerate}[label=\roman*)]
 \item $Q_1(p) = f(p)\,, \forall\, p \in P\cap H$,
 \item $Q_2(p) = (f(p) - Q_1(p))/Q_H(p)\,, \forall\, p \in P\setminus H$.
\end{enumerate}
Then, $Q = Q_1 +Q_HQ_2 \in \Pi$ is the unique polynomial in $\Pi$ that interpolates $f$ in $P$, i.e., $Q(p)=f(p) \,\, \forall\, p \in P$.
\label{theorem:PIP}
\end{theorem}

\begin{figure}[t!]
\center
\includegraphics[scale=0.15]{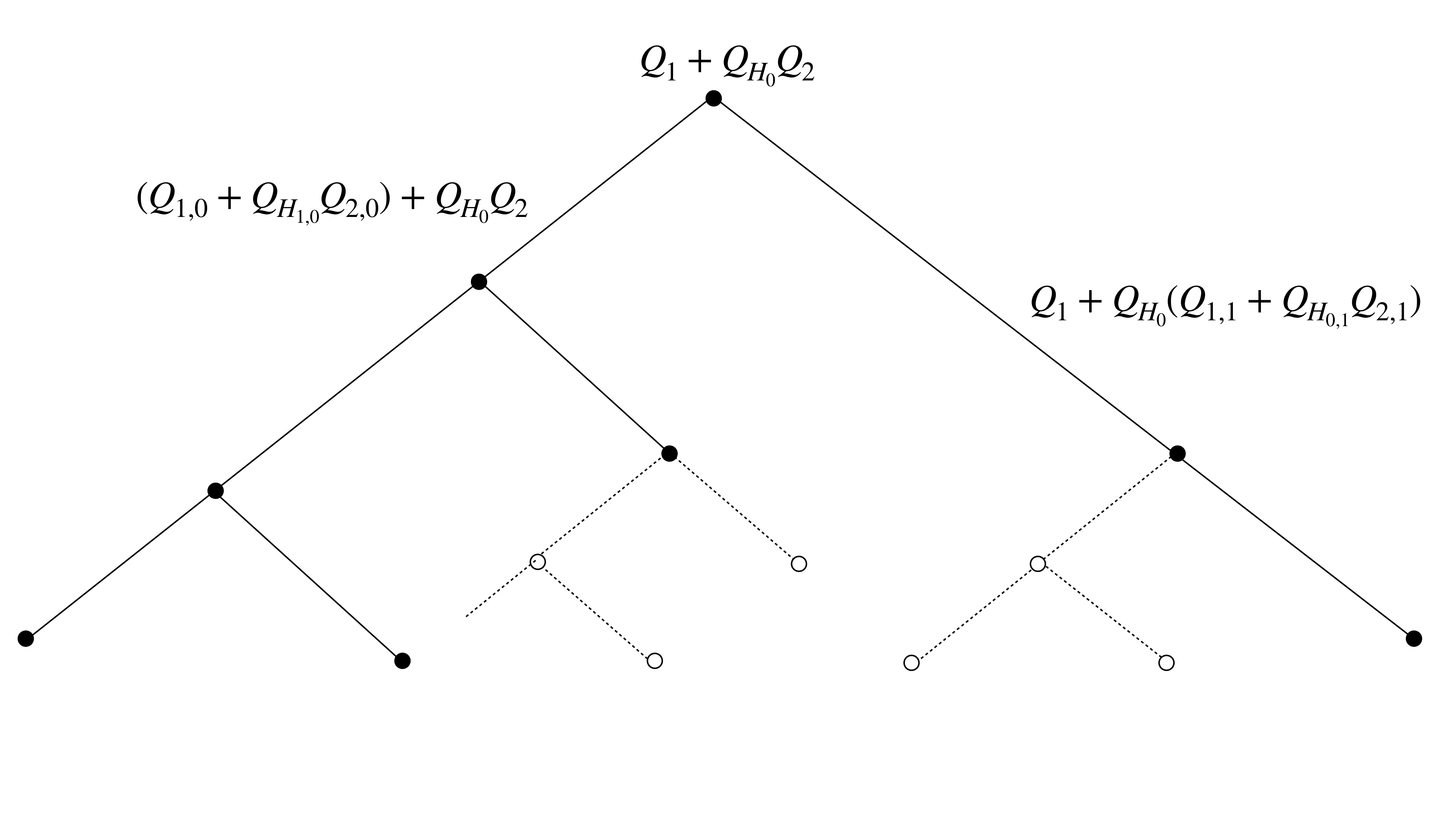}
\caption{The generalised divided difference scheme given by recursively choosing suitable hyper(sub)planes $H_{0}, H_{1,0}, H_{0,1}, \ldots$ and nodes $P=P_0 =P_{1,0}\cup P_{0,1}, \ldots$ according to Theorem~\ref{theorem:UN}, and applying the splitting Theorem~\ref{GS} to the separated polynomials $Q_{1},Q_2,Q_{1,0},Q_{2,0},Q_{1,1},Q_{2,1},\ldots$. \label{Fig:TREE}}
\end{figure}

\begin{proof} $Q_H\not = 0$ on $\R^m \setminus H$ implies that $Q(p)=f(p)$, $\forall\, p \in P$. Thus, $Q$ interpolates $f$ in $P$.
To show the uniqueness of $Q$ let $Q' \in \Pi$ interpolate $f$ in  $P$.
Then, $Q-Q' \in \Pi$ and $(Q-Q')(p) = 0 \, \, \forall p \in P$.  Due to Theorem \ref{theorem:UN},  $P$ is unisolvent with respect to
$\Pi$. Thus, $Q'-Q\equiv 0$ is the zero polynomial, proving that $Q$ is uniquely determined in $\Pi$.
\end{proof}

\begin{remark} Recursion of Theorem~\ref{GS} yields a general divided difference scheme, as
illustrated in Fig.~\ref{Fig:TREE} \citep[see also][]{PIP1,IEEE}, which requires evaluating $Q_1$in $P \setminus H$ in each recursion step.
Unless the computational costs of these evaluations can be reduced to liner time $\Oc(|P\setminus H|)$, one ends up with super-quadratic, up to cubic $\Oc(|A|^3)$, runtime, as in \citep{cohen2,cohen3}.
Choosing unisolvent nodes as non-tensorial grids (cf.~Fig.~\ref{Fig:UN}, right panel), however, avoids the $Q_1$-evaluation, resulting in an interpolation algorithm with quadratic runtime complexity $\Oc(|A|^2)$.
\end{remark}

\subsection{Unisolvent non-tensorial grids}
As a direct consequence of Theorem~\ref{theorem:UN}, we deduce:

\begin{corollary}\label{cor:grid} Let $m \in \N$, $A\subseteq \N^m$ be a downward closed set of multi-indices, and $\Pi_A \subseteq \Pi_m$ the polynomial sub-space induced by $A$.
Let $P_i = \{p_{0,i},\dots,p_{n_i,i}\} \subseteq \square_1$ be arbitrary sets of size $n_i \geq \max_{\alpha \in A} \alpha_i$. Then, the node set
\begin{equation}\label{eq:LCL}
  P_A = \li\{ (p_{\alpha_1,1}\,, \dots \,, p_{\alpha_m,m} ) : \alpha \in A\re\}
\end{equation}
is unisolvent with respect to $\Pi_A$.
\end{corollary}

\begin{proof} We argue by induction on $m$ and $|A|$. For $m=1$  the claim follows from the fact that $\dim \Pi_A= |A|$ and no polynomial $Q\in \Pi_A$ can vanish in  $|A|$ distinct nodes
$P_A$. The claim becomes trivial for $|A|=1$. Now assume that $m >1$ and $|A|>1$. We consider $A_1=\li\{\alpha \in A :  \alpha_m =0\re\}$, $A_2 = A \setminus A_1$.
By decreasing $m$ if necessary and w.l.o.g., we can assume that $A_2 \not = \emptyset$. Consider the hyperplane
$H =\{ (x_1,\dots,x_{m-1},p_{0,m}):  (x_1,\ldots,x_{m-1}) \in \R^{m-1}\}$ and
$Q_H \in \Pi_{m,1,1}$ with $Q_H(x)= x_m-p_{0,m}$. Induction yields that $P_A$ is unisolvent with respect to $(\Pi_A,H)$.
By realising that the affine translation
$\tau_H(x) = (x_1,\ldots,x_{m-1},x_m)- (0,\ldots,0,p_{0,m})$ is canonical with respect to $\Pi_A$, Theorem~\ref{theorem:UN} applies and proves $P_A$ to be unisolvent with respect to $\Pi_A$.
\end{proof}

\begin{figure}[t!]
\center
\includegraphics[scale=0.15]{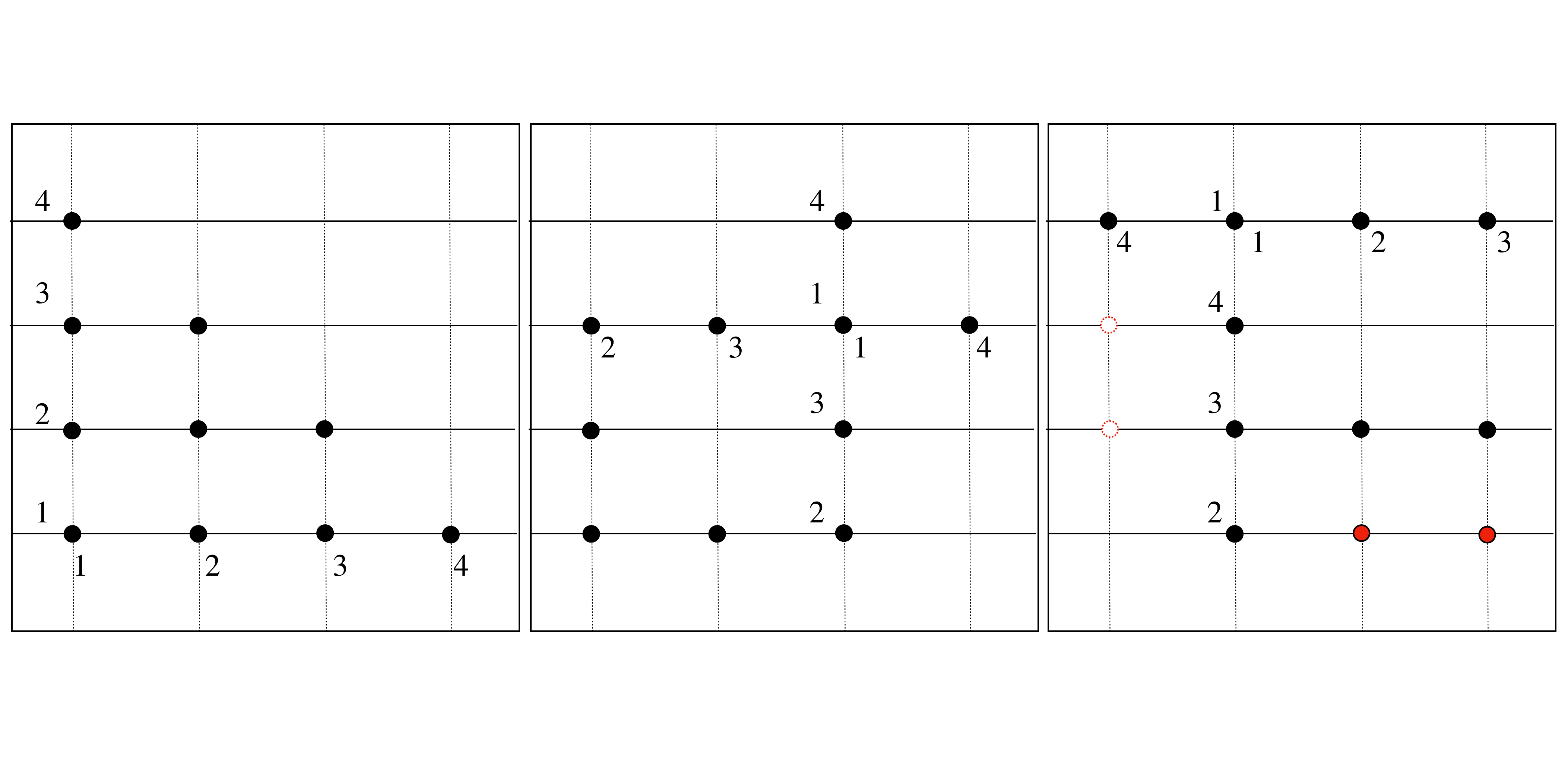}
\caption{Examples of unisolvent nodes for $A= A_{2,3,1}$ (left, middle) and $A_{2,3,2}$ (right). Note that $(2,2) \in A_{2,3,2}\setminus A_{2,3,1}$ generates an extra node. Orderings in $x,y$--directions are indicated by numbers, and non-tensorial nodes are shown in red. \label{Fig:UN2}}
\end{figure}

\begin{example}
It is important to note that although the index sets $A$ are assumed to be downward closed,
the flexibility in ordering the $P_i$ results in unisolvent nodes $P_A$ that may induce \emph{non-tensorial (non-symmetric) grids}, in which there are nodes $p=(p_x,p_y) \in P_A$ with $(p_y,p_x)\not \in P_A$. This might occur even when all $P_i =P$, $1 \leq i \leq m$, only differ by reordering. Examples are shown in Fig.~\ref{Fig:UN} (right), Fig.~\ref{Fig:UN2} (right), and Fig.~\ref{Nodes}. In comparison, Fig.~\ref{Fig:UN2} (left, middle) shows examples of symmetric  grids, which occur if, in addition, all $P_i$ coincide in their ordering.
\end{example}

We continue by showing that efficient interpolation is possible in non-tensorial grids.

\section{Multivariate Newton interpolation in non-tensorial grids}

We provide a natural extension of the classic Newton interpolation scheme to arbitrary dimensions, directly based on the notion of unisolvence, Theorem~\ref{theorem:UN}, and Corollary~\ref{cor:grid}. This completes previous contributions \citep{neidinger2019}, which did not guarantee unique interpolants apart from the total- and maximum-degree case.
Examples and less formal explanations are given in the documentation of the accompanying Python package {\sc minterpy}~\citep{minterpy}.

\subsection{Multivariate Newton interpolation}\label{sec:NEWT}

The extension relies on recursively applying Theorem \ref{theorem:PIP} and Corollary \ref{cor:grid}. We start by defining:
\begin{definition}[Multivariate Newton polynomials] Let $A\subseteq \N^m$ be downward closed and $P_A\subseteq \square_m$ unisolvent nodes as in Corollary~\ref{cor:grid}.
We define the \emph{multivariate Newton polynomials} as:
 \begin{equation}\label{Newt}
  N_\alpha(x) = \prod_{i=1}^m\prod_{j=0}^{\alpha_i-1}(x_i-p_{j,i}) \,, \quad \alpha \in A\,.
 \end{equation}
\end{definition}

In dimension $m=1$, this reduces to the classic 1D Newton polynomials~\citep[see, e.g.,][]{Stoer,gautschi,trefethen2019}. In $m$D the notion allows interpolating by a divided difference scheme:

\begin{definition}[Multivariate divided differences] \label{def:DDS}  Let $A\subseteq \N^m$ be downward closed, $P_A\subseteq \square_m$ unisolvent nodes as in  Corollary~\ref{cor:grid},
and $f : \R^m \lo \R$ a function.
For $\alpha =(a_1,\dots,a_m)\in A$ we define $\beta^{\alpha,i,j} = (b_1,\dots,b_m)\in A$, with $\leq j  <a_i$, $i=1,\dots,m$, as:
\begin{equation}\label{BETA}
b_h = \li\{\begin{array}{ll}
                        a_h & \,\,\, \text{if}\,\,\, h \not =  i  \\
                        j   & \,\,\, \text{if}\,\,\, h = i \, .
                        \end{array}\re.
\end{equation}
Then, we recursively define the {\em multivariate divided differences}:
$$
F_{\alpha,m,0}  =f(p_{\alpha}) \,, \quad    F_{\alpha,i,0}  =F_{\alpha,i+1,a_{i+1}-1}\ \quad  \text{for}\,\,\,   1 \leq i < m
$$
and
\begin{equation}\label{Falpha}
    F_{\alpha,i,j}:=\frac{F_{\alpha,i,j-1} - F_{\beta^{\alpha,i,j-1},i,j-1}}{p_{a_i,i}-p_{j-1,i} }
\quad \text{for}\,\,\, 1 \leq j\leq a_i \,.
\end{equation}
Finally, we define
$c_{\alpha}:= F_{\alpha,1,a_1}$, $\alpha =(a_1,\dots,a_m) \in A$, as the \emph{Newton coefficients} of $Q_{f,A} \in \Pi_A$.
\label{MVDD}
\end{definition}

In dimension $m=1$, this definition recovers the classic \emph{divided difference scheme of 1D Newton interpolation}~\citep{Stoer,gautschi}.
In $m$D, we state:

\begin{theorem}[Multivariate Newton interpolation] \label{theo:MIP} Let the assumptions of Definition~\ref{def:DDS} hold.
Then, the unique polynomial $Q_{f,A} \in \Pi_A$ interpolating $f$ in $P_A$, $Q_f(p) = f(p), \,  \forall \, p \in P_A$, can be determined in $\Oc(|A|^2)$ operations requiring $\Oc(|A|)$ storage. It is given by
\begin{equation}\label{Newton}
  Q_{f,A}(x) = \sum_{\alpha \in A} c_\alpha N_{\alpha} (x)\,,
\end{equation}
where $c_\alpha$ are the \emph{Newton coefficients} of $Q_{f,A} \in \Pi_A$.
\end{theorem}

\begin{proof} Since the statement is classic for $m=1$, we assume $m >1$ and argue by induction on $|A|$. For $|A|=1$ the claim follows immediately.
For  $|A|>1$ we consider $A_1=\li\{\alpha \in A : \alpha_m =0\re\}$, $A_2 = A \setminus A_1$. By decreasing $m$ if necessary and
w.l.o.g.,~we can assume that $A_2  \not = \emptyset$.
Consider the hyperplane
$H =\{(x_1,\dots,x_{m-1},p_{0,m}): (x_1,\dots,x_{m-1}) \in \R^{m-1}\} = Q_H^{-1}(0)$, on which $Q_H(x) = x_m - p_{0,m} \in \Pi_{m,1,1}$, and the canonical transformation $\tau_H : \R^m\lo\R^m$ with $\tau_H(x) = (x_1,\dots,x_m) - (0,\dots,0, p_{0,m})$,
$\tau_H(H) = \R^m \times\{0\}$. Let $\pi_{m-1} : \R^m \lo \R^{m-1}$, $\pi_{m-1}(x_1,\dots,x_m)=(x_1,\dots,x_{m-1})$, be the natural projection and $i_{m-1} : \R^{m-1} \hookrightarrow \R^m$, $(x_1,\dots,x_{m-1}) \mapsto  (x_1,\dots,x_{m-1},0)$,
be the natural inclusion.

{\bf Step 1:} We reduce the interpolation to $H$. We set  $P_1 = \pi_{m-1}\big(\tau_H(P_A\cap H)\big)$ and  $f_0 : \R^{m-1} \lo \R$ with
\begin{equation}\label{ftau}
 f_0(x_1,\dots,x_{m-1}) = f\big(\tau_H^{-1}(i_{m-1}(x_1,\dots,x_{m-1}))\big)=f\big(x_1,\dots,x_{m-1},p_{0,m})\,.
\end{equation}
Let $M_{\alpha}(x) \in \Pi_{A_1}$, $\alpha \in A_1$, be the Newton polynomials with respect to $A_1$, $P_1$. Induction yields that the coefficients $d_\alpha \in \R$ of the unique polynomial
$$ Q_{f_0,A_1}(x_1,\dots,x_{m-1}) = \sum_{\alpha \in A_1 } d_\alpha M_{\alpha}(x_1,\dots,x_{m-1})$$
interpolating $f_0$ in  $P_1$ can be determined in less than $D_0|A_1|^2$ operations, $D_0 \in \R^+$, requiring a linear amount of storage. The Newton polynomials $N_{\alpha} \in \Pi_A$, $\alpha \in A_1$, are given by $i_{m-1}^*\big(\tau_{H}^{*}(N_\alpha)\big) = M_\alpha$. Thus, $N_\alpha(x_1,\dots,x_m) = M_{\alpha}\big(i_{m-1}(\tau_{H}(x_1,\dots,x_{m}))\big) = M_{\alpha}(x_1,\dots,x_{m-1})$.
We set
\begin{equation}\label{poly1}
Q_1(x_1,\dots,x_m) := Q_{f_0,A_1}(x_1,\dots,x_{m-1}).
\end{equation}
Then,
$Q_1(x_1,\dots,x_m)= \sum_{\alpha\in A_1}d_{\alpha}N_{\alpha}(x_1,\dots,x_m)$ satisfies $Q_1(p) =f(p)$, $\forall p \in P_A \cap H$.

{\bf Step 2:} By definition, $Q_1$ is constant in direction $x_m$, i.e., $Q_1(x_1,\dots,x_{m-1},y) =Q_{f_0,A_1}(x_1,\dots,x_{m-1})$ for all $y \in \R$. Further, each
$\alpha \in A_2$ is given as $\alpha= \beta + (0,\dots,0,i)$ for exactly one $\beta = \beta^{\alpha,m,0} \in A_1$, $i \in \N$, as in Eq.~\eqref{BETA}. Thus, Eq.~\eqref{ftau} implies
$Q_1(p_{\alpha}) = f(p_{\beta^{\alpha,m,0}})$.
Setting $f_1(x) = (f(x) - Q_1(x))/Q_H(x)$, it then requires $D_1|A_2|$, $D_1 \in \R^+$ operations to compute all values $F_{\alpha,1,m}$ from Eq.~\eqref{Falpha} due to
\begin{equation*}
f_1(p_{\alpha})=\frac{f(p_{\alpha}) - Q_1(p_{\alpha})}{Q_H(p_{\alpha})}=\frac{f(p_{\alpha}) - f(p_{\beta^{\alpha,m,0}})}{p_{\alpha_m,m}-p_{0,m}} =
\frac{F_{\alpha,m,0}- F_{\beta^{\alpha,m,0},m,0}}{p_{\alpha_m,m}-p_{0,m}} =F_{\alpha,m,1} \quad \text{for all} \,\, \alpha \in A_2\,.
\end{equation*}
We set $\widetilde A_2 = A_2 - e_m$, $e_m=(0,\dots,0,1)\in \N^m$, and $P_{\widetilde A_2} =\{\tilde p_\gamma\}_{\gamma \in \widetilde A_2 }$ with $\tilde p_\gamma = p_{\gamma + e_m} \in P_2$
for all $\gamma \in  \widetilde A_2$.
Denote by $K_{\gamma}(x) \in \Pi_{\widetilde A_2}$ the Newton polynomials with respect to $\widetilde A_2, P_{\widetilde A_2}$. Then,
induction yields that the coefficients
$b_\gamma \in \R$, $\gamma \in \widetilde A_2$, of the unique polynomial
$$Q_2(x_1,\dots,x_m):=Q_{f_1, \widetilde A_2}(x_1,\dots,x_m)= \sum_{\gamma \in \widetilde A_2}b_\gamma K_{\beta}(x_1,\dots,x_m)$$
interpolating $f_1$ in  $P_2=P_{\widetilde A_2}$ can be determined in less than $D_0|A_2|^2$ operations, requiring linear storage. Due to Eq.~\eqref{Newt}, we observe that  $Q_H(x)K_\gamma(x) = N_{\gamma+e_m}(x)$ for all $\gamma \in \widetilde A_2$. While $P_A$ is unisolvent due to Corollary \ref{cor:grid},
Theorem \ref{theorem:PIP} implies that
the unique polynomial $Q\in \Pi_A$ interpolating $f$ in $P_A$ is given by:
\begin{equation}
  Q_{f,A}(x) =Q_1(x) + Q_H(x)Q_2(x)
  =\sum_{\alpha \in A_1}d_\alpha N_{\alpha}(x) + Q_H(x)\sum_{\gamma\in \widetilde A_2}h_\gamma K_{\beta}(x)  =  \sum_{\alpha \in A}c_\alpha N_{\alpha}(x)\,, \label{recurs}
\end{equation}
where $c_\alpha = d_\alpha$ for $\alpha \in A_1$ and $c_\alpha = h_{\alpha-e_m}$ for $\alpha \in A_2$. Due to Definition~\ref{def:DDS}, recursion of this inductive argument yields that $c_\alpha = F_{\alpha,\alpha_1,1}$ $\forall \alpha \in A$.
In total, the computation can hence be done in less than $D_0|A_1|^2 + D_1|A_2|+D_0|A_2|^2\leq \max\{D_0,D_1\}(|A_1|+|A_2|)^2 \in \Oc(|A|^2)$ operations and $\Oc(|A_1| + |A_2|))=\Oc(|A|)$ storage.
\end{proof}

Theorem \ref{theo:MIP} implies that every polynomial $Q \in \Pi_A$ can be uniquely expanded as $Q=\sum_{\alpha \in A}c_{\alpha}N_{\alpha}$, meaning that
  the Newton polynomials$ \{N_{\alpha}\}_{\alpha \in A} \subseteq \Pi_A$
  are a basis of $\Pi_A$. While evaluating a multivariate polynomial in the canonical basis requires finding a suitable factorisation of a multivariate Horner scheme \citep[see, e.g.,][]{Stoer,gautschi,Jannik},
evaluation and differentiation are straightforward in Newton basis:

\begin{theorem}[Evaluation and differentiation in Newton basis] \label{theo:DIFF} Let $A\subseteq \N^m$ be downward closed, $P_A\subseteq \square_m$ unisolvent nodes as in  Corollary~\ref{cor:grid}, $Q(x) = \sum_{\alpha \in A}c_\alpha N_{\alpha} \in \Pi_A$,
$c_\alpha \in \R$, a polynomial in Newton basis, and $x_0 \in \R^m$. Then:
\begin{enumerate}[label=\roman*)]
 \item\label{C1} there is a recursive algorithm requiring $\Oc(|A|)$ operations and $\Oc(|A|)$ storage to evaluate $Q$ at $x_0$;
 \item\label{C2} there is an iterative algorithm requiring $\Oc(m|A|)$ operations and $\Oc(|A|)$ storage to evaluate $Q$ at $x_0$;
 \item\label{C3} there is an iterative algorithm requiring  $\Oc(nm|A|)$ operations and $\Oc(|A|)$ storage to evaluate the partial derivative $\partial_{x_j}Q$, $1 \leq j \leq m$, at $x_0$.
\end{enumerate}
\end{theorem}

\begin{proof}
To prove \ref{C1} we follow the proof of Theorem \ref{theo:MIP} using induction over the number of coefficients. Due to Eq.~\eqref{recurs}, $Q_1$ and $Q_2$
can be evaluated in linear time. Since the evaluation of $Q_H(x)= x_m-p_{0,m}$ requires constant time, the claim follows. To show \ref{C2}, we observe that computing and storing the values of the products
$q_{i,k} = \prod_{j=0}^{k}(x_{0,i}-p_{j,i})$, $x_0 = (x_{0,1},\dots,x_{0,m}) \in \R^m$, $i =1,\dots m$, $k =1,\dots,n$,
requires $\Oc(mn)$ operations. Then
\begin{equation}
 Q(x_0) = \sum_{\alpha \in A} c_{\alpha}N_{\alpha}(x_0) = \sum_{\alpha \in A} c_{\alpha} \prod_{i=1}^m q_{i,\alpha_{i}-1}
\end{equation}
is computable in $\Oc(m|A|)$ operations. Because $|A| \geq mn$, this yields \ref{C2}. Similarly, the partial derivative
\begin{equation}
 \partial_{x_j}Q(x_0) = \sum_{\alpha \in A} c_{\alpha} \prod_{i=1, i\not = j}^m q_{i,\alpha_{i}} \sum_{h=0}^{\alpha_j-1} \hat q_{j,h} \,, \quad \hat q_{j,h} =\prod_{l =0 , l \not =h}^{\alpha_j-1}(x_{0,j}-p_{l,j})\,.
\end{equation}
Hence we obtain \ref{C3}, and the theorem is proven.
\end{proof}

\begin{remark} The recursive splitting $Q=Q_1+Q_HQ_2$ from Eq.~\eqref{recurs},  appearing in \ref{C1},  recovers the classic Aitken-Neville algorithm in dimension $m=1$~\citep{neville}.
We further note that although requiring $\Oc(m|A|)$ runtime, numerical experiments suggest that the iterative algorithm in \ref{C2} is faster in practice, while maintaining the (machine-precision) accuracy achieved by the recursive algorithm from \ref{C1}~\citep{PIP2}.
\end{remark}

\subsection{Multivariate Lagrange interpolation} \label{sec:LAG}

We can also use the above concepts to extend
1D Lagrange interpolation~\citep{berrut} and tensorial $m$D Lagrange interpolation \citep{Gasca,sauerL,sauertens,trefethen2019} to the case of $m$D non-tensorial unisolvent nodes.

\begin{definition}[Lagrange polynomials] \label{def:LagP}
 Let  $m \in \N$, $A \subseteq \N^m$ be a downward closed set of multi-indices, and let $P_A = \{p_{\alpha}\}_{\alpha \in A}$ be a unisolvent set of nodes with respect to the polynomial space $\Pi_{A}$.
We define the \emph{multivariate Lagrange polynomials}
\begin{equation}\label{DL}
  L_{\alpha} \in \Pi_{P_A}\ \quad \text{with}\quad L_{\alpha}(p_\beta)= \delta_{\alpha,\beta}\, , \,\,\, \alpha,\beta \in A\,,
\end{equation}
where $\delta_{\cdot,\cdot}$ is the Kronecker delta.
\end{definition}

\begin{corollary}[Lagrange basis]\label{cor:Lag} Let the assumptions of Definition \ref{def:LagP} hold. Then:
\begin{enumerate}
 \item[i)]  the Lagrange polynomials $L_\alpha \in \Pi_A$ are a basis of $\Pi_A$;
 \item[ii)] the polynomial $Q_{f,A}(x) = \sum_{\alpha \in A}f(p_{\alpha})L_{\alpha}(x) \in \Pi_{A}$ is the unique polynomial interpolating $f$ in  $P_A$, and it can be determined in $\Oc(|A|)$ operations.
\end{enumerate}
\end{corollary}

\begin{proof} To show $i)$, we note that there are $|A|$ Lagrange polynomials and $\dim \Pi_A = |A|$. Given
$c_{\alpha} \in \R$, $\alpha \in A$, such that
$\sum_{\alpha \in A} c_\alpha L_\alpha = 0$, the unisolvence of $P_A$ implies that the polynomial $Q(x)= \sum_{\alpha \in A} c_\alpha L_\alpha$ vanishes in  $P_A$ and, therefore, has to be the zero polynomial.
Hence, $c_\alpha =0$ for all $\alpha \in A$,
implying that the $L_\alpha \in \Pi_A$ are linearly independent and thus a basis of $\Pi_A$. The formula and the uniqueness in $ii)$ then follow from $i)$.
\end{proof}

\begin{remark}
In the maximum-degree case, $A= A_{m,n,\infty}$, the grid $P_A$ becomes tensorial, and the above definition recovers the known  \emph{tensorial $m$D Lagrange interpolation}:
\begin{equation}\label{LagTens}
  L_{\alpha}(x)= \prod_{i=1}^m l_{\alpha_i,i} (x) \,, \quad l_{j,i} (x) = \prod_{h=0, h \not = j}^n \frac{x_i-p_{h,i}}{p_{\alpha_i,i} - p_{h,i}}  \,,
\end{equation}
where $x = (x_1,\dots,x_i,\dots,x_m) \in \R^m$, $1 \leq i,j\leq n$, $\alpha \in A$.
Using the multivariate Newton interpolation from Theorem~\ref{theo:MIP} with $f = L_{\alpha}$, an explicit expression for the Lagrange polynomials can be derived even in the general case of a downward closed $A\subseteq \N^m$ and non-tensorial grid $P_A$:
\begin{equation}\label{LN}
 L_{\alpha}(x) = \sum_{\beta \in A}c_{\alpha,\beta} N_\beta(x)\,, \quad c_{\alpha,\beta} \in \R\,.
\end{equation}
Thus, Theorem~\ref{theo:DIFF} provides for efficient evaluation and differentiation of the Lagrange interpolant.
\end{remark}

Of course, the question arises which among the
possible unisolvent node sets to choose when aiming to maximize the interpolant's approximation power. We consider this question in the next section.

\section{Leja-ordered nodes and Lebesgue constants}

\begin{figure}[t!]
\center
\includegraphics[scale=1.05]{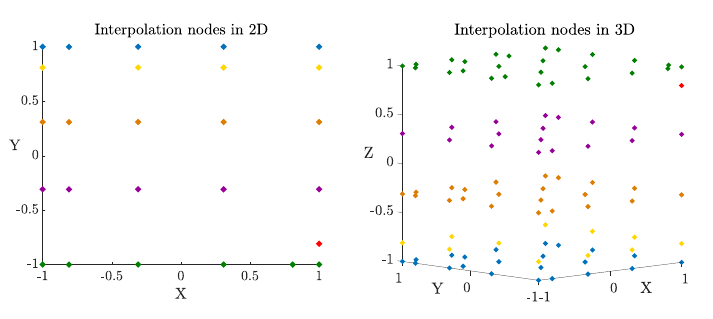}
\caption{Leja-ordered Chebyshev-Lobatto (LCL) nodes $P_A$, $A=A_{m,5,2}$ in 2D ($m=2$, left) and 3D ($m=3$, right).
Nodes in the same horizontal hyperplane are coloured equally.} \label{Nodes}
\end{figure}

The crucial contribution of \cite{Fekete1923} to interpolation and potential theory \cite{Taylor2008, Bos2010} is the notion of \emph{Fekete points}. We recall:

\begin{definition}[Fekete points]
    \label{def:Fekete}
    Let \(1 \leq k \leq n\) and $P_n = \{p_0, p_1, \ldots, p_n\}$ be a set of $n+1$ points. \emph{Fekete points} of order \(k\) are defined as \(k+1\) distinct points
    \[
        F_k \coloneqq \{p_0, p_1, \ldots, p_{k}\} \subset P_n = \{p_0, p_1, \ldots, p_n\}
    \]
    that maximizes the absolute value of the Vandermonde determinant
    \begin{equation}
        \label{eq:determinant-maximization}
        V(p_0,\ldots,p_{k}) \coloneqq \det ((p_i)^j)_{i,j=0,\ldots,k}.
    \end{equation}
    Thus, the Fekete points are (not uniquely) determined by
    \[
        F_k \in \mathrm{argmax}_{x_0,\ldots,x_{k} \in P} |V(x_0,\ldots,x_{k})| = \mathrm{argmax}_{x_0,\ldots,x_{k} \in P} \prod_{0 \leq i<j\leq k} |x_j-x_i|\,.
    \]
\end{definition}
 Fekete points are among the best choices for polynomial interpolation, which is reflected by the following fact:   Let \(\Lambda(F_k)\) denote the Lebesgue constant of Fekete points of order \(k\), and let \(\Lambda(P_n)\) denote the Lebesgue constant of the full set of nodes \(P\). Then
    \begin{equation}
        \label{eq:Lebesgue-Fekete}
        \Lambda(F_k) \leq (k+1)\, \Lambda(P), \quad 0 \leq k \leq n\,,
    \end{equation}
    see \citep{bos2018bernstein}.
However, there are two drawbacks of Fekete points. Firstly, the computation requires  \(\mathcal{O}\big(n^2{n+1 \choose k+1}\big)\) operations.
Secondly, they are not necessarily nested, i.e.,
\[
	F_k \not\subset F_{k+1}, \quad 0 \leq k \leq n\,,
\]
a property that allows bounding the Lebesgue constant for interpolation in downward closed polynomial spaces.
We propose two alternative choices of nodes, relaxing the notion of Fekete points and (partially) overcoming the stated issues.
To do so, we revisit the concept of Leja points \citep{leja}.
\begin{definition}[Leja-ordered points]\label{EA}
    Let \(K \subseteq \R\) be a compact set and $LP_n = \{p_0,\ldots,p_n\} \subseteq K$ such that
    \[
        |p_0| = \max_{p \in K} |p|, \quad \prod_{j=0}^{l-1} |p_l - p_j| = \max_{p \in K} \prod |p - p_j|, \quad p_l \in K, \quad 1 \leq l \leq n.
    \]
    then $LP_n$ are called Leja points \citep{leja} with respect to $K$ or shortly Leja points for $K=\square_1$. In case where $K=P_n$ is a set of cardinality $n+1$, we call the resulting ordered set $P_n^{\mathrm{Leja}_\geq}$ Leja ordered and specifically denote
    \begin{equation*}
        \Cheb_n^{\mathrm{Leja}_\geq} = \li\{ \cos\Big(\frac{k\pi}{n}\Big) :  0 \leq k \leq n\re\}^{\mathrm{Leja}_\geq}
    \end{equation*}
    in case of $P_n = \Cheb_n$.

Let $A\subseteq \N^m$, $m \in \N$, be downward closed.
Generating non-tensorial grids $P_A$ (see Corollary~\ref{cor:grid}) from $P_i = LP_n$ or $\Cheb_n^{\mathrm{Leja}_\geq}$ yields the \emph{Leja points (LP nodes)} or \emph{Leja-ordered Chebyshev-Lobatto nodes (LCL nodes)}, respectively.
\end{definition}

Examples of LCL nodes $P_A \subseteq \square_m$ are shown in Fig.~\ref{Nodes} in 2D (left) and 3D (right). Similarly, Leja-ordered versions of Fekete or Legendre nodes can be used to generate unisolvent grids $P_A$. While in 1D, the ordering of the points has no influence on the Lebesgue constant, the situation changes in $m$D.
The following bounds on the Lebesgue constants are crucial for proving the approximation rates of BLT-function interpolation, as formulated in \ref{CC2}.

\begin{lemma}\label{lemma:LEB} Let $P_n = \{p_0,\dots,p_n\}\subseteq [-1,1]$, $|P_n|=n+1$, $n \in \N$ be a set of nodes. We denote with $P_{h}=\{p_0,\dots,p_{h}\}$, $h < n$ the truncation to the first $h+1$ nodes.
\begin{enumerate}[label=\roman*),left=0pt]
\item\label{L1}
Let $A\subseteq \N^m$, $m \in \N$,  be downward closed and $P_A \subseteq \square_m$ be unisolvent nodes generated by
$P_i \subseteq \square_1$, $i=1,\dots,m$, according to  Corollary~\ref{cor:grid}. Then
the Lebesgue constant
$$\Lambda(P_A) = \sup_{f \in C^0(\square_m), \|f\|_{C^0(\square_m)} \leq 1 }\|\mathcal{Q}_{P_A}f\|_{C^0(\square_m)}\,,$$
given as the operator norm of the interpolation operator $\mathcal{Q}_{P_A} : C^{0}(\square_m) \to \Pi_A \subseteq C^{0}(\square_m)$, $ f \mapsto Q_{f,P_A}$, scales as
\begin{equation}
  \Lambda(P_A) = \sup_{x \in \square_m} \sum_{\alpha \in A} |L_\alpha(x)| = \Oc\big(|A|^{\theta+1}\big) \,,
\end{equation}
whenever $\Lambda(P_{i,h_i}) \leq (n_i +1)^\theta$, $\forall \, 0 \leq h_i \leq n_i=|P_{i}|$, $i=1,\dots,m$ and some $\theta \geq 1$.
In the maximum-degree case $A = A_{m,n,\infty}$,  $\Lambda(P_A) = \Oc(\prod_{i=1}^m\Lambda(P_i))$.
\item \label{L2} When extending $\mathcal{Q}_{P_A} : C^{k}(\square_m) \to \Pi_A \subseteq C^{k}(\square_m)$ up to $k$-th order derivatives, $k \in \N$, the $k$-th order Lebesgue constant
$\Lambda(P_A)_k = \sup_{f \in C^k(\square_m), \|f\|_{C^k(\square_m)} \leq 1 }\|\mathcal{Q}_{P_A}f\|_{C^k(\square_m)}$
is bounded by
\begin{equation}\label{LEB}
  \Lambda(P_A)_k = \sum_{\beta \in A_{m,k,1}}\sup_{x \in \square_m} \sum_{\alpha \in A} |\partial_\beta L_\alpha(x)| = \Oc\Big(\binom{m+k}{k}|A|^{\theta +2k+1}\Big)\,,
\end{equation}
whereas $\Oc(\binom{m+k}{k}n^{2k}\prod_{i=1}^m\Lambda(P_i))$
applies in the maximum-degree case $A = A_{m,n,\infty}$.
\end{enumerate}
\end{lemma}
\begin{proof}
\ref{L1} follows from the estimates of \cite{cohen2}. By a standard tensorial argument \citep{HOSQ}, the maximum-degree case follows.
 \ref{L2} relies on Markov's inequality \citep{markov1889problem}, stating that the derivative of any polynomial $p_n \in \Pi_{1,n,1}$, $n \in \N$, is bounded by $p_n$ itself as $\|p_n'\|_{C^0(\square_1)} \leq n^2 \|p_n\|_{C^0(\square_1)}$. Recalling $|A_{m,k,1}| = \binom{m+k}{k} \in \Oc(m^k)\cap \Oc(k^m)$ yields the stated bound.
\end{proof}

We note:
\begin{enumerate}[left=0pt,label=\textbf{N\arabic*)}]
\item When choosing LCL nodes, $P_i = \Cheb_n^{\mathrm{Leja}_\geq}$, $i=1,\dots,m$, we have the estimate $\Lambda(\Cheb_n)=\frac{2}{\pi}\big(\log(n+1) + \gamma +  \log(8/\pi)\big) + \Oc(1/n^2)$,
where $\gamma \approx  0.5772$ is the Euler-Mascheroni constant \citep{brutman2, trefethen2019}. However, though the Lebesgue constant of the truncations of
$\Cheb_n^{\mathrm{Leja}_\geq}$ seem to be be bounded linearly, no explicit bound is known.

\item For the nested LP nodes, $P_i =LP_n$, $i=1,\dots,m$, algebraic bounds of $\Lambda(LP_n) = \Oc(n^{13/4})$ apply \citep{LejaLEB}.

\item Based on Eq.~\eqref{LN}, we numerically measure the Lebesgue constants at 10,000 random points.
Figure~\ref{LEB1-3} shows that Lebesgue constants for the total and Euclidean degrees $A = A_{m,n,p}$, $p=1,2$ grow sub-exponentially, whereas those for the maximum degree (blue lines) follow the estimated logarithmic scaling.
\end{enumerate}

\begin{figure}[htp!]
\center
\includegraphics[scale=0.4]{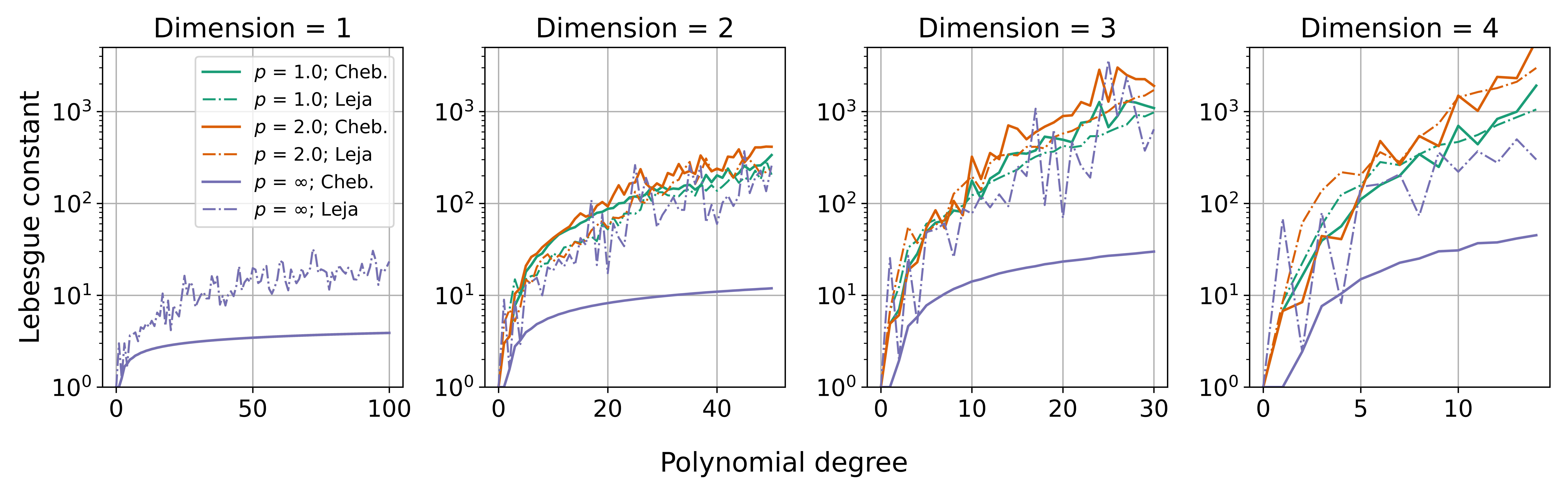}
\caption{Numerically measured Lebesgue constants for LCL nodes (solid lines) and LP nodes (dashed lines) $P_A \subseteq \square_m$ for total ($p=1$, green), Euclidean ($p=2$, red), and
maximum degree ($p=\infty$, blue) polynomial interpolants in dimensions $m =1,\dots,4$ (panels from left to right).} \label{LEB1-3}
\end{figure}

Beside the theoretical gap of bounding the Lebesgue constants of LCL nodes, the results suggest LCL and LP nodes to deliver similar approximation power.
With these ingredients we state:

\begin{theorem}\label{theo:APP} Let $m,n \in \N$, $p >0$, $A =A_{m,n,p} \subseteq \N^m$, and  $f : \square_m \lo \R$ be a BLT-function with approximation rate
\begin{equation*}
      \|f - p_n^*\|_{C^0(\square_m)}  \lesssim \rho^{-n}\,,  \quad \rho = \rho_p >1 \,,
  \end{equation*}
  where $ p_{n}^* \in \Pi_{m,n,p}$ denotes the best approximation. Given unisolvent nodes $P_A\subseteq \square_m$, satisfying \ref{L1} of Lemma~\ref{lemma:LEB} (e.g, LP nodes). Then:
  \begin{enumerate}[label=\roman*),left=0pt]
      \item\label{T1} the approximation error of the interpolant $Q_{f,P_A}$ of $f$ is bounded by
      \begin{equation*}
        \| f - Q_{f,P_A}\|_{C^0(\square_m)}  \lesssim (1 + \Lambda(P_A))\rho^{-n} = \Oc_{\ee}(\rho^{-n})\,,
      \end{equation*}
      where $\Lambda(P_A)$ is given in Lemma~\ref{lemma:LEB}\ref{L1}, respectively.
      \item\label{T2} for any order $k \in \N$, the derivatives  interpolants approximate the derivatives of $f$ with
      \begin{equation*}
        \| f- Q_{f,P_A}\|_{C^k(\square_m)} \lesssim (1 +\Lambda(P_A)_k) \rho^{-n} = \Oc_{\ee}(\rho^{-n})\,,
      \end{equation*}
     where $\Lambda(P_A)_k$ is given in Lemma~\ref{lemma:LEB}\ref{L2}.
  \end{enumerate}
\end{theorem}

\begin{proof} Point \ref{T1} follows from Lemma~\ref{lemma:LEB} together with the classic Lebesgue inequality
\begin{align*}
    \| f - Q_{f,P_A}\|_{C^0(\square_m)}   \leq  \| f - p_n^*\|_{C^0(\square_m)} +  \Lambda(P_A) \| f - p_n^*\|_{C^0(\square_m)}\,,
  \end{align*}
where  $p_n^*\in \Pi_{m,n,p}$ denotes the best approximation.

To show \ref{T2},
we note that, by Theorem~\ref{theo:BL}, the function $f =F_{|\square_m}$ is
the restriction of a function $F$ holomorphic in $\Omega_{\rho(P,K)} \supset K=\square_m$, $P =A_{m,n,p}$, $\rho>1$ and so are all of its derivatives. Hence, $f$ and all of its derivatives are BLT-functions w.r.t.~the same domain $\Omega_{\rho(P,K)}$. Consequently, and
analogously to \ref{T1}, the bound follows from substituting the estimate for the corresponding $k$-th order Lebesgue constant from Lemma~\ref{lemma:LEB}\ref{L2}.
\end{proof}

We next verify these statements in numerical experiments, confirming in particular that in the Euclidean-degree case LCL-node interpolants perform better than LP ones, but the derivatives of LP-node interpolants reach similar or better approximation rates as those of LCL-node interpolants.

\section{Numerical experiments}\label{sec:NUM}

We experimentally verify our results using a MATLAB prototype named {\sc MIP}, implementing the multivariate divided difference scheme from Definition \ref{def:DDS}.
Then, we benchmark an optimised \emph{open-source Python implementation}, called {\sc minterpy}~\citep{minterpy}.

The experimental results reported in section~\ref{sec:MIP} indicate that MIP resist the curse of dimensionality best among the tested state-of-the-art alternatives. The results of section~\ref{sec:minterpy} validate Theorem~\ref{theo:APP}, showing that our approach can reach optimal approximation rates. Even if the achieved Euclidean-degree rates are mostly lower than the maximum-degree ones, differences are small.
Especially in higher dimensions, this renders Euclidean-degree interpolation the standard choice w.r.t.~\ref{Q2}. Interpolation on LCL and LP nodes performs mostly comparably in the Euclidean-degree case. However, only LCL grids achieve optimal rates, while LP nodes appear to be the better choice when evauating derivatives of the interpolants.

\subsection{{\sc MIP} benchmarks}\label{sec:MIP}
{\sc MIP} uses LCL nodes for Euclidean degree ($p=2$), and we compare it with the following alternative methods:
\begin{enumerate}[left=0pt,label=\textbf{B\arabic*)}]
\item {\sc chebfun} from the corresponding MATLAB package \citep{chebfun};
\item {\sc cubic splines} and {\sc $5^{th}$-order splines} from the MATLAB \emph{Curve Fitting Toolbox};
\item {\sc Floater-Hormann} interpolation \citep{floater} from {\sc chebpol} \citep{gaure};
\item {\sc multi-linear} (piecewise linear) interpolation  from  {\sc chebpol} \citep{gaure};
\item {\sc Chebyshev} interpolation of $1^\text{st}$ kind from  {\sc chebpol} \citep{gaure};
\end{enumerate}
Apart from {\sc MIP}, all other methods use tensorial grids as interpolation nodes. {\sc chebfun} and {\sc Chebyshev} only deliver $l_\infty$-degree interpolants.
The interpolation degree of Floater-Hormann and all spline interpolations is set as $n = \mathrm{argmax}_{n \in \N} \{ C \geq |A_{m,n,\infty}|\}$, being the largest maximum degree that results in a smaller set of coefficients than the total number of coefficients, $C \in \N$, the corresponding method requires.
All implementations are benchmarked using MATLAB version R2019b, {\sc Chebfun} package version 5.7.0, and R version 3.2.3/Linux.
The code and all benchmark data sets are available at
\emph{https://git.mpi-cbg.de/mosaic/polyapprox}.

\begin{experiment}\label{exp1}
We measure the approximation errors of the interpolants computed by the tested methods for the Runge function,  $f(x) = \frac{1}{s^2 + r^2\|x\|^2}$ (Eq.~\eqref{eq:Runge}), with $r^2 =1,10$, $s=1$, resulting in the optimal rates $\rho>1$ as reported in Table~\ref{Trate}.
To measure the approximation errors $\| f- Q_f\|_{C^0(\square_m)}$, we measure the $L_\infty$-error at $100$ random points $M\subseteq \square_m$, $|M|=100$. These points are sampled {\it i.i.d.} for each degree, but are identical across methods.
The approximation rates of {\sc MIP} are fitted with the model $y = c\rho_\mathrm{MIP}^{-n}$ with a $R$-squared of $0.99$ or better, as reported in
Table~\ref{Trate}.
\end{experiment}


%
%
%

Figure \ref{Rate_3D} shows the results in dimension $m=3$.
We observe that {\sc Floater-Hormann} is indistinguishable from {\sc $5^{th}$-order splines}.
When considering the number of coefficients/nodes required to determine the interpolant, plotted in the right panel, the polynomial convergence rates of {\sc Floater-Hormann} and all spline-type approaches become apparent. On the contrary, with a slight advantage over {\sc Chebyshev} and {\sc chebfun}, \textsc{MIP} nearly reaches the optimal exponential convergence rate from Eq.~\eqref{eq:rate}. Hereby, \textsc{MIP} requires $122^3/899028 \approx 2$ times fewer coefficients/nodes than {\sc Chebyshev} or {\sc chebfun} require to approximate $f$ to machine precision for degree $n=121$.

\begin{figure}[t!]
\center
 \includegraphics[scale=0.28]{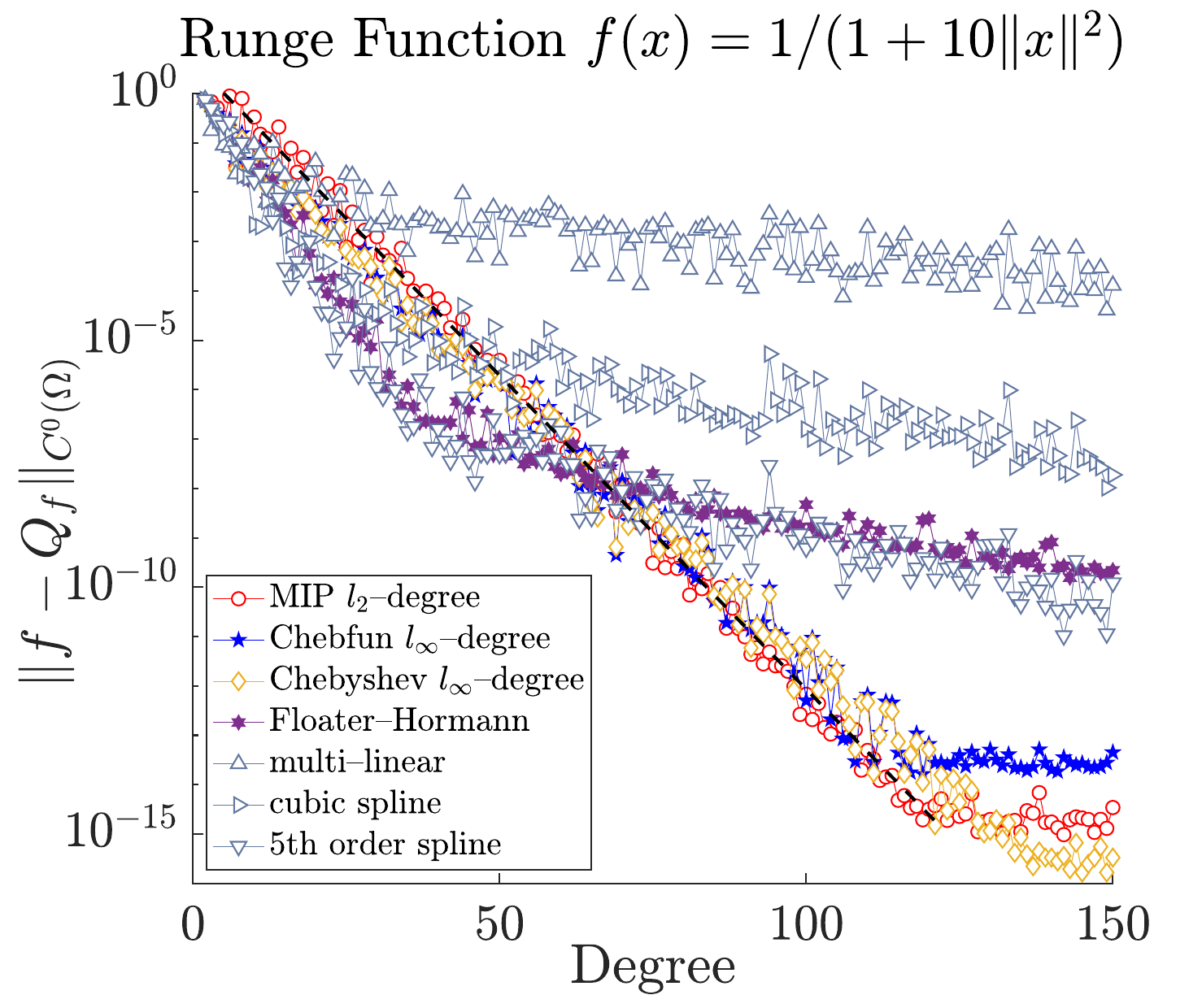}
  \includegraphics[scale=0.28]{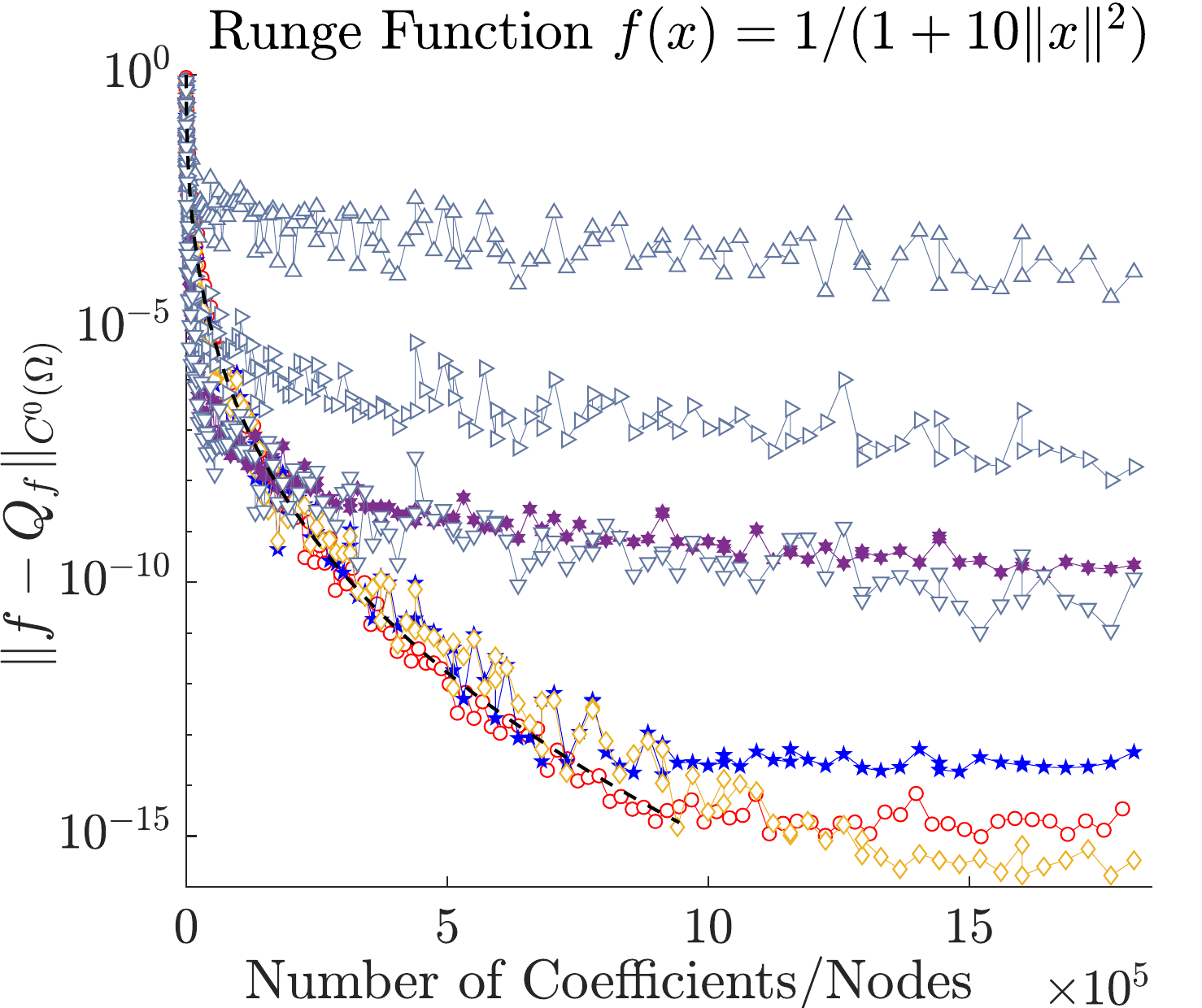}
\caption{Approximation errors for the benchmarked methods interpolating the Runge function in dimension $m=3$. The fitted asymptotic rate of {\sc MIP} from Table \ref{Trate} is indicated by the black dashed line. } \label{Rate_3D}
\end{figure}

\begin{figure}[t!]
\center
\includegraphics[scale=0.28]{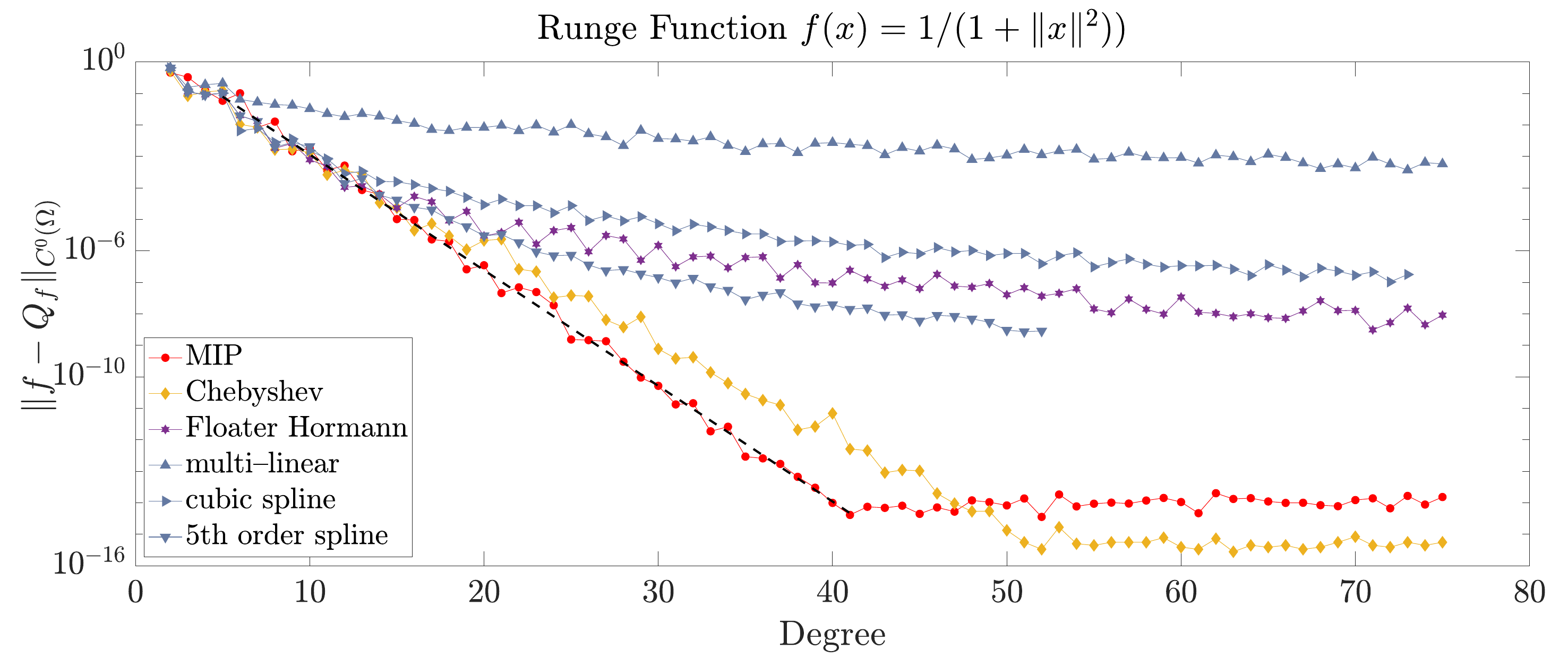}
\caption{Approximation errors for the benchmarked methods interpolating the Runge function in dimension $m=4$. The fitted asymptotic rate of {\sc MIP} from Table \ref{Trate} is indicated by the black dashed line.} \label{Rate_4D}
\end{figure}
\begin{table}[t!]
\center
\begin{tabular}{lcccc}
function &  dim   & $\rho_{\mathrm{MIP}}$ & $c$ & $\rho$   \\
 \hline
$f(x) = 1/(1+10\|x\|^2)$ &  3 &  $1.34$ & $4.41$  & $1.365$\\
$f(x) = 1/(1+1\|x\|^2)$  &  4 &  $2.33$ & $5.40$  & $2.41$ \\
$f(x) = 1/(1+1\|x\|^2)$  &  5 &  $2.35$ & $13.37$ & $2.41$ \\
\end{tabular}
\caption{Fitted approximation rates $\rho_{\mathrm{MIP}}$ for {\sc MIP} with respect to the model $y = c\rho_\mathrm{MIP}^{-n}$, compared with the
theoretical optimal rates $\rho>1$ from Eq.~\eqref{eq:Rbounds}.
\label{Trate}}
\end{table}
\begin{figure}[t!]
\center
  \includegraphics[scale=0.28]{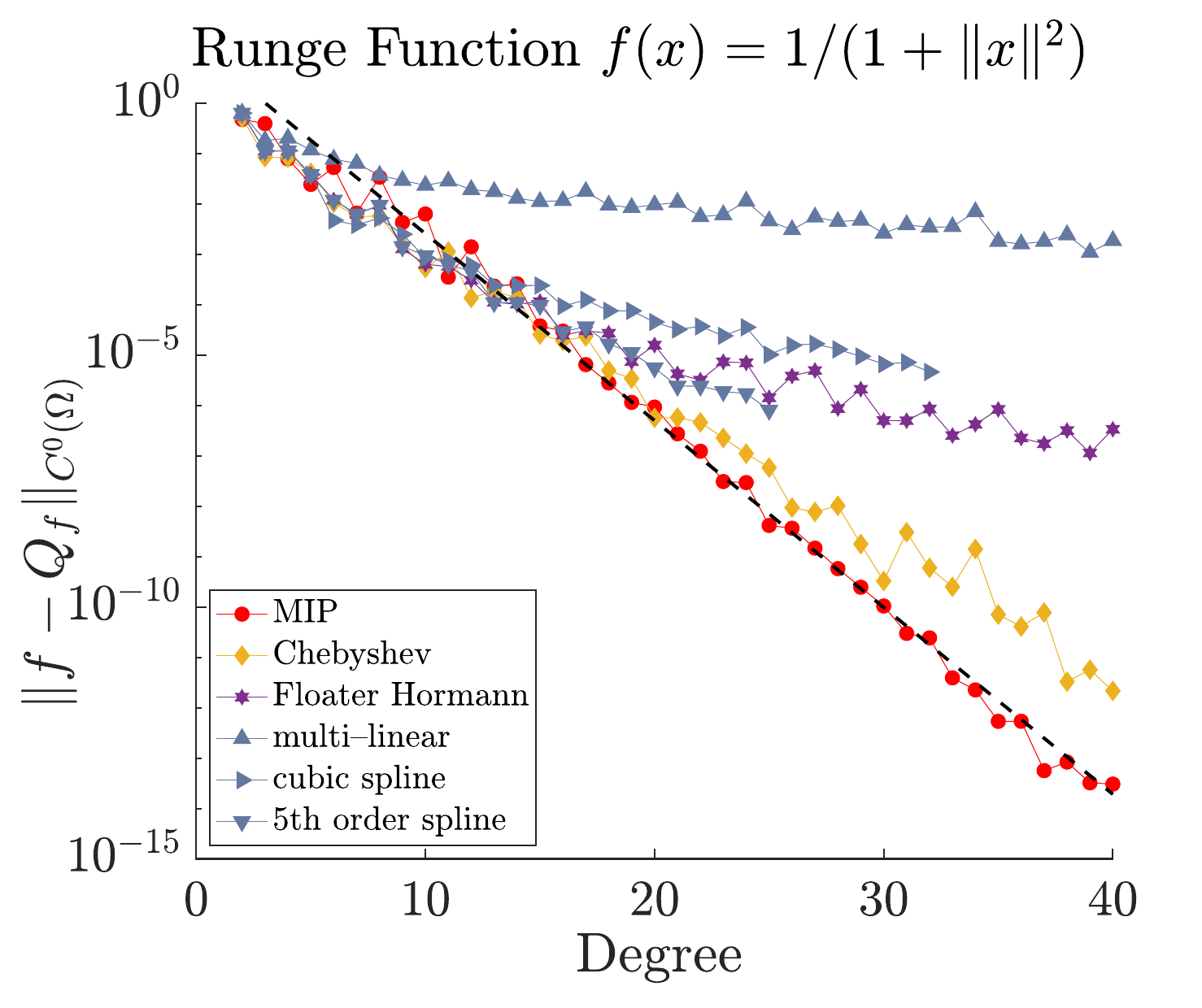}
  \includegraphics[scale=0.28]{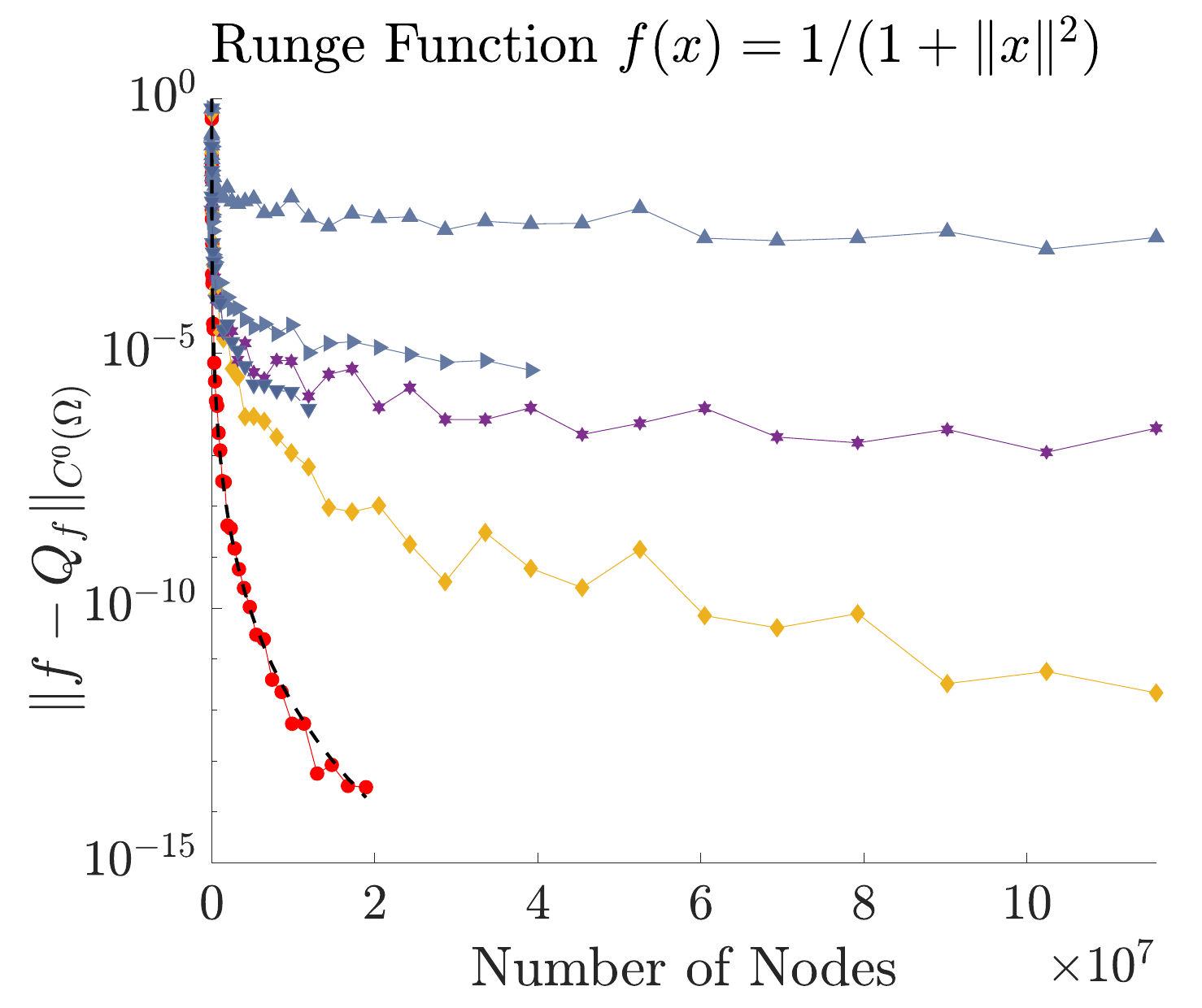}
\caption{Approximation errors for the benchmarked methods interpolating the Runge function in dimension $m=5$. The fitted asymptotic rate of {\sc MIP} from Table \ref{Trate} is indicated by the black dashed line.} \label{Rate_5D}
\end{figure}

Figure \ref{Rate_4D} shows the results in dimension $m=4$. Here, spline interpolation is unable to scale to high degrees, due to computer memory requirements.
Therefore, we interpolate the simpler Runge function $f(x) =\frac{1}{1+\|x\|^2}$, i.e. $r^2=1$. 4D interpolation is not supported in {\sc chebfun} and thus not given here.
Only {\sc Chebyshev} and {\sc MIP} converge to machine precision, with {\sc MIP} converging at nearly the optimal rate. Consequently,
{\sc MIP} reaches machine precision earlier than {\sc Chebyshev}, namely for degree $n=40$ ({\sc Chebyshev}: $n=47$).

In dimension $m=5$, the advantage of {\sc MIP} over {\sc Chebyshev} further increases, as shown in Fig.~\ref{Rate_5D}. {\sc MIP} best resists the curse of dimensionality, yielding two orders of magnitude better accuracy than {\sc Chebyshev} for $n=40$ (error $3.0\cdot 10^{-14}$ vs.~$3.2\cdot 10^{-12}$). Again, {\sc MIP} does so requiring fewer interpolation nodes $\frac{|C_{\mathrm{Chebyshev}}|}{|C_{\mathrm{MIP}}|} = \frac{115856201}{18920038} \approx 6$.
An error of $4.0\cdot 10^{-12}$ is even reached by {\sc MIP} with 10 times fewer ($7.4\cdot10^6$ vs.~$7.9\cdot 10^7$) interpolation nodes than {\sc Chebyshev}.

\subsection{{\sc minterpy} benchmarks}\label{sec:minterpy}

We continue the numerical experiments by presenting benchmarks for the Python {\sc minterpy}~\citep{minterpy} implementation of the multivariate divided difference scheme from Definition \ref{def:DDS}, as well as efficient evaluation and differentiation according to Theorem~\ref{theo:DIFF}. We numerically compute the 1D Leja points using the Python function {\sc scipy.optimize}.

\begin{experiment}\label{exp2}
We revisit Experiment~\ref{exp1}, now measuring the $L_\infty$ errors of the LCL-node and LP-node $l_p$-degree interpolants for $p=1,2,\infty$, and their derivatives, at $10^6$  random points, {\it i.i.d.} for each degree, but identical across  methods.
The approximation rates are fitted with the model $y = c\rho^{-n}$ with a $R$-squared of $0.99$ or better. The solid lines in the plots show the resulting fits.
We consider several functions with optimal rates $\rho$ mostly known, thanks to \cite{bos2018bernstein}:
\begin{enumerate}[left=0pt,label=\textbf{F\arabic*)}]
\item\label{F1} The bivariate function
\begin{equation}
    f(x_1,x_2) = \frac{1}{(x_1-r)^2 + x_2^2}\,, \,\, 1 < r \in \R\,, \quad\text{with}\quad  \rho_p =  \begin{cases}
    r & ,\, p = 1 \\
    r-1 +\sqrt{(r-1)^2 +1} & ,\, p = \infty
  \end{cases}\,,
\end{equation}
 Though not proven in \citep{bos2018bernstein}, the Euclidean and maximum degree rates are expected to coincide.
 \item\label{F2} The multivariate Runge function
 $$f(x) = \frac{1}{1 + r^2\|x\|^2}\,, \quad 1< r \in \R\,,$$
 from Eq.~\eqref{eq:Runge} with optimal rates according to Eq.~\eqref{eq:Rbounds}.
 \item\label{F3} The multivariate perturbed Runge function
\begin{equation*}
    f(x) = \frac{1}{1 + (\sum_{i=1}^m r_i x_i)^2}\,,\quad r_i =5/i^3\,.
\end{equation*}

\item\label{F4} The multivariate extension of the function from Eq.~\eqref{eq:BLT},
 $$f(x) = \frac{1}{\sum_{i=1}^m (x_i - a)^2}\,, \quad 1< a \in \R\,.$$
\item\label{F5} The  entire multivariate trigonometric function
\begin{equation*}
    f(x) = \cos(\pi \mathbf{k_1}\cdot x) + \sin(\pi \mathbf{k_2}\cdot x)\,, \quad \mathbf{k_1},\mathbf{k_2} = k_1\mathbf{1},  k_2\mathbf{1} \in \N^m\,, \mathbf{1} =(1,\dots,1) \in \N^m\,, k_1,k_2 \in \N\,,
\end{equation*}
with unbounded asymptotic rate $\rho >1$.
\end{enumerate}
\end{experiment}

\begin{figure}[!htp]
    \centering
    \captionsetup{width=1\linewidth}
    \vspace{-1em}
    \begin{subfigure}[t]{0.48\linewidth}
        \centering
        \includegraphics[scale=0.3]{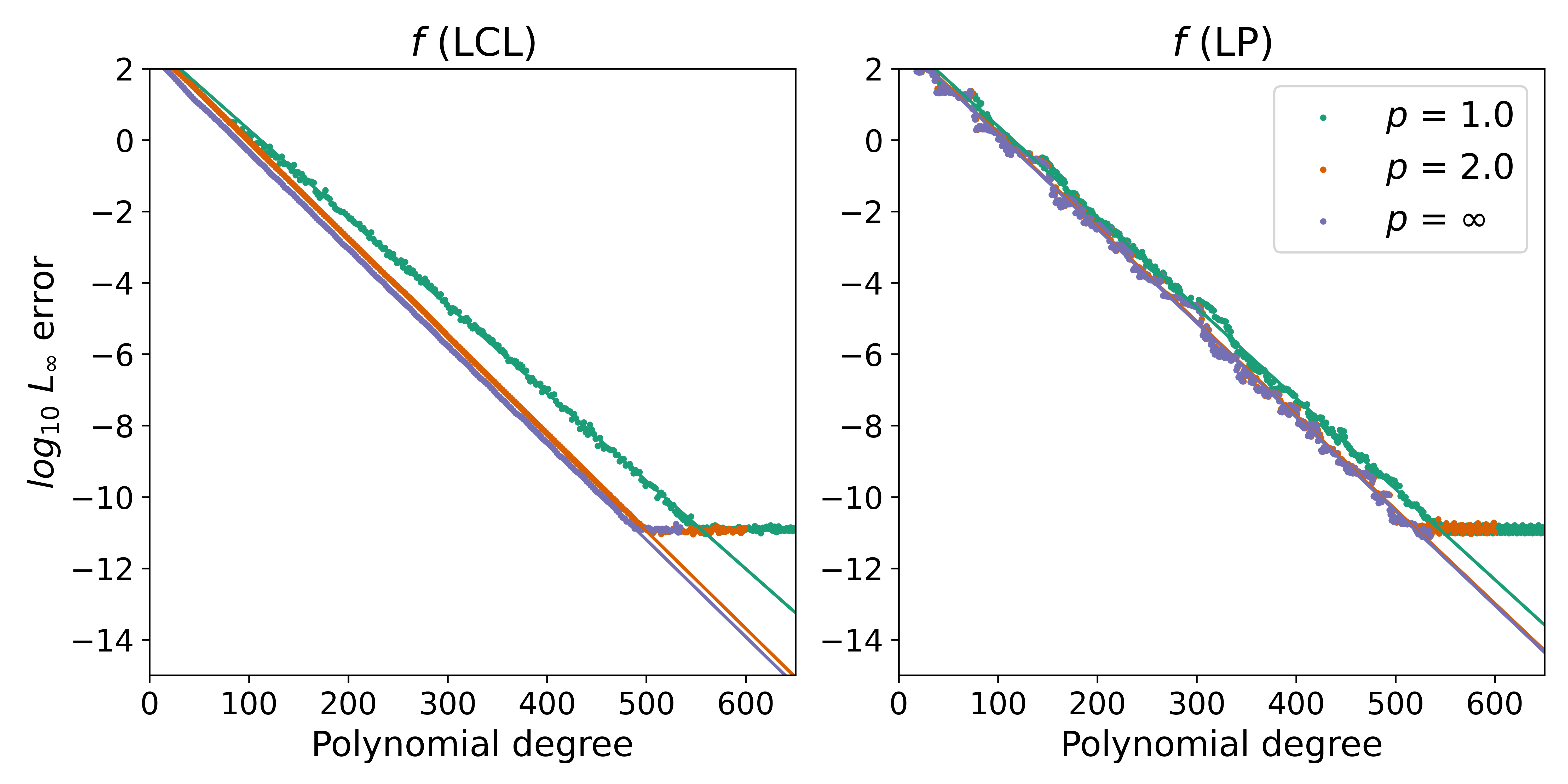}
        \caption{\ref{F1} for $r = 17/16$}
    \end{subfigure}
    \hfill
    \begin{subfigure}[t]{0.48\linewidth}
        \centering
        \includegraphics[scale=0.315]{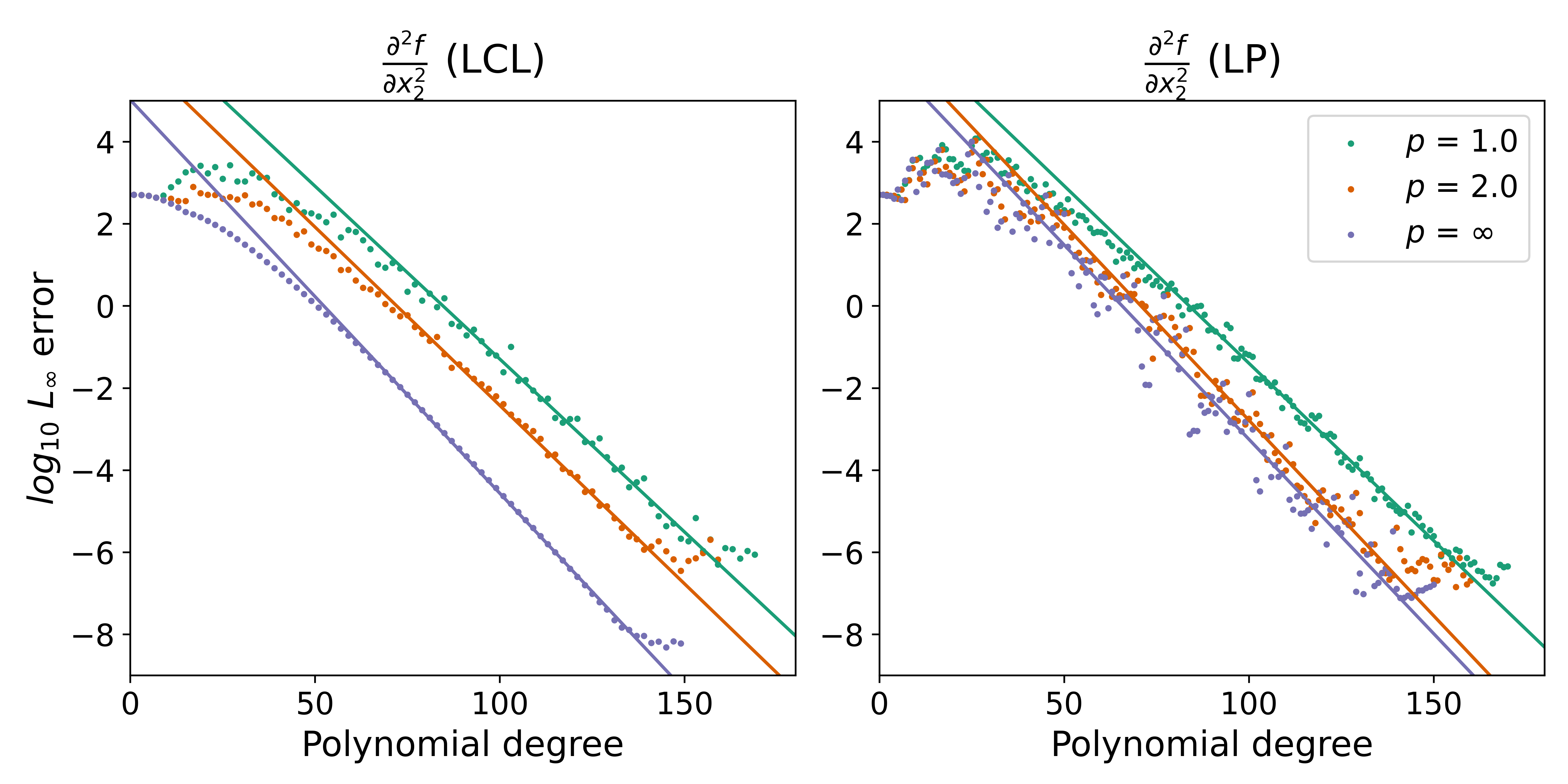}
        \caption{$2$nd derivative of \ref{F1} for $r = 5/4$}
    \end{subfigure}
    \caption{Approximation errors of the LCL- and LP-node interpolants of \ref{F1} and their $2$nd derivative in dimension $m = 2$ for $p=1,2,\infty$. Solid lines show the fitted rates from Tables~\ref{tab:diff-F1} and \ref{tab:F1}.\label{SC-F1}}
    \label{fig:F1}
\end{figure}

\begin{table}[!htp]
    \centering
    \captionsetup{width=1.0\linewidth}
    \begin{subtable}[t]{0.45\textwidth}
        \centering
        \begin{tabular}{c|c|cc}
            \toprule
            \multirow{2}{*}{$r$} &
            \multirow{2}{*}{$p$} &
            \multicolumn{2}{c}{$f$} \\
            \cline{3-4}
            &
            & LCL & LP \\
            \midrule
            & $1.0$    & $1.119$ & $1.122$ \\
            $\frac{9}{8}$ & $2.0$    & $\mathbf{1.133}$ & $1.128$ \\
            & $\infty$ & $\mathbf{1.133}$ & $1.128$ \\
            \bottomrule
        \end{tabular}
    \end{subtable}
    \hspace{1cm} 
    \begin{subtable}[t]{0.45\textwidth}
        \centering
        \begin{tabular}{c|c|cc}
            \toprule
            \multirow{2}{*}{$r$} &
            \multirow{2}{*}{$p$} &
            \multicolumn{2}{c}{$f$} \\
            \cline{3-4}
            &
            & LCL & LP \\
            \midrule
            & $1.0$    & $1.058$ & $1.060$ \\
            $\frac{17}{16}$ & $2.0$    & $\mathbf{1.065}$ & $1.063$ \\
            & $\infty$ & $\mathbf{1.065}$ & $1.063$ \\
            \bottomrule
        \end{tabular}
    \end{subtable}
    \caption{Fitted approximation rates of LCL- and LP-node interpolants of \ref{F1} for different $r$ and $p$. Bold indicates cases in which the optimal rate was actually achieved.}
    \label{tab:F1}
\end{table}

\begin{table}[!t]
    \centering
    \captionsetup{width=1.0\linewidth}
    \begin{subtable}[t]{1.0\textwidth}
     \centering
 \begin{tabular}{c|c|cc|cc|cc}
    \toprule

    \multirow{2}{*}{$r$} &
    \multirow{2}{*}{$p$} &
    \multicolumn{2}{c}{$f$}&
    \multicolumn{2}{c}{$\frac{\partial f}{\partial x_1}$} &
    \multicolumn{2}{c}{$\frac{\partial^2 f}{\partial x^2_1}$}\\
    \cline{3-8}
    &
    &
    LCL &
    LP &
    LCL &
    LP &
    LCL &
    LP\\
    \midrule

               & $1.0$    & $1.243$ & $1.243$  & $1.228$ & $1.224$ & $1.212$ & $1.216$ \\
$\frac{5}{4}$  & $2.0$    & $1.280$ & $1.277$ & $1.260$ & $1.256$ & $1.229$ & $1.241$\\
               & $\infty$ & $\mathbf{1.281}$ & $1.274$  & $1.263$ & $1.254$ & $1.252$ & $1.246$\\
\toprule

    \multirow{2}{*}{} &
    \multirow{2}{*}{} &
    \multicolumn{2}{c}{$\frac{\partial f}{\partial x_2}$} &
    \multicolumn{2}{c}{$\frac{\partial^2 f}{\partial x^2_2}$}&
        \multicolumn{2}{c}{}  \\
    \cline{3-8}
    &
    &
    LCL &
    LP &
    LCL &
    LP &
     &
    \\
    \midrule
& $1.0$ &
$1.235$ &
$1.245$ &
$1.214$ &
$1.219$  & & \\

$\frac{5}{4}$ & $2.0$ &
$1.251$ &
$1.262$ &
$1.222$ &
$1.245$ & & \\

& $\infty$ &
$1.266$ &
$1.267$ &
$1.247$ &
$1.244$ & & \\
\bottomrule
 \end{tabular}
    \end{subtable}
 \caption{Fitted approximation rates of LCL- and LP-node interpolants of the $1$st and $2$nd derivatives of \ref{F1} for different $r$ and $p$ in dimension $m = 3$.
 Bold indicates cases in which the optimal rate was actually achieved.}
 \label{tab:diff-F1}
\end{table}

\subsubsection{Discussion of results for \ref{F1}}
The approximation errors of the interpolants and their derivatives are plotted in Fig.~\ref{SC-F1}. The fitted approximation rates are reported in Tables~\ref{tab:diff-F1} and \ref{tab:F1}.

As theoretically predicted, the approximation rates of Euclidean- and maximum-degree interpolants coincide, with LCL-node interpolants performing superior to LP-node interpolants, reaching the optimal rate.
This maintains true also in the total-degree case, where LP-node interpolants are slightly closer to the optimal rate $\rho =r$,  $r \in\{ 5/4 =1.250, 9/8 = 1.125, 17/16 = 1.0625\}$ than LCL-node interpolants.

The derivatives reach, as expected, slightly lower rates (Table \ref{tab:diff-F1}).  LCL- and LP-node interpolants perform comparably for the $1$st derivatives. However,  for the $2$nd derivatives in the Euclidean- and total-degree cases, LP-node interpolants are superior to LCL-node ones.

\begin{figure}[t!]
\vspace{-0.5em}
    \centering
    \captionsetup{width=1.0\linewidth}

    \begin{subfigure}[t]{0.48\linewidth} 
        \centering
        \includegraphics[scale=0.3]{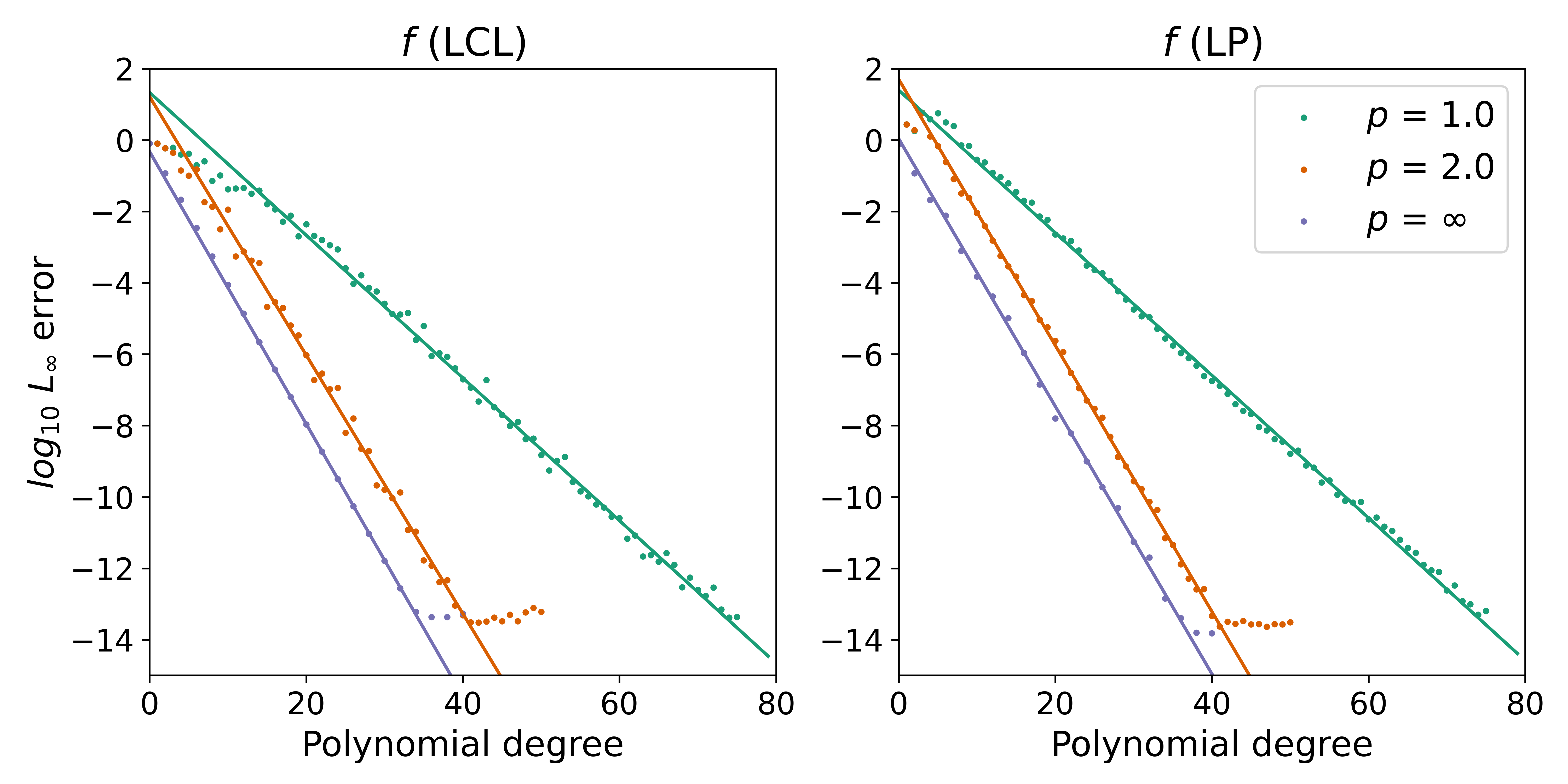}
        \caption{\ref{F2} in dimension $m = 4$ with $r = 1$}
        \label{subfig:cheb-leja}
    \end{subfigure}
    \hfill
    \begin{subfigure}[t]{0.48\linewidth} 
        \centering
        \includegraphics[scale=0.31]{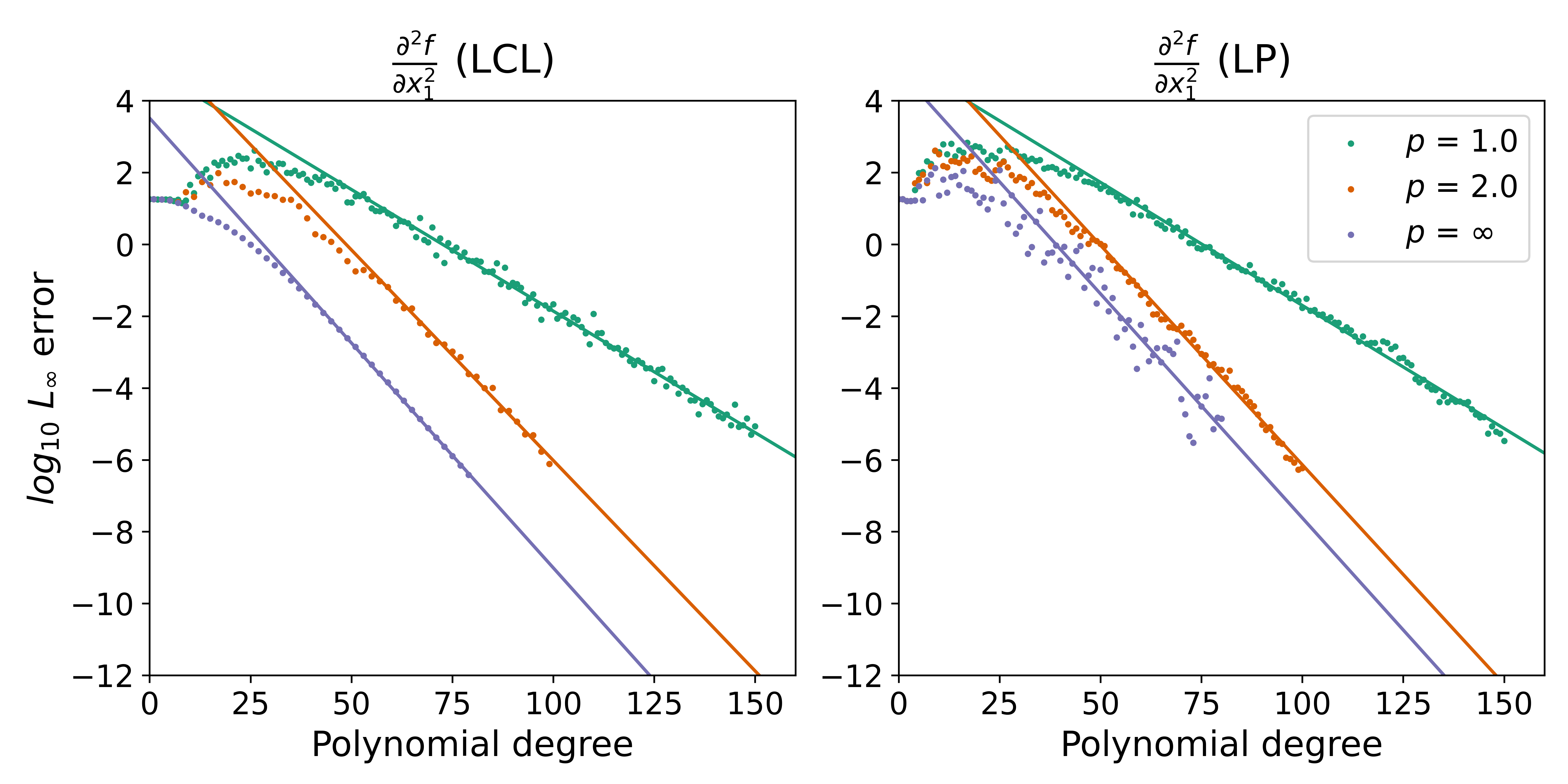}
        \caption{$2$nd derivative of \ref{F2} in dimension $m = 3$ with $r = 3$}
        \label{subfig:diff}
    \end{subfigure}
        \caption{Approximation errors of the LCL- and LP-node interpolants of \ref{F2} and their $2$nd derivative for $p=1,2,\infty$. Solid lines show the fitted rates from Tables~\ref{tab:2} and \ref{tab:diff-runge}.}
    \label{Rungerates}
\end{figure}

\begin{table}[!htp]
\vspace{-1.0em}
    \centering
    \captionsetup{width=1.0\linewidth}
    \begin{subtable}[t]{0.45\textwidth}
        \centering
        \begin{tabular}{c|c|cc}
            \toprule
            \multirow{2}{*}{$r$} & \multirow{2}{*}{$\rho$} &
            \multicolumn{2}{c}{$f$} \\
            \cline{3-4}
            &
          &
            LCL &
            LP \\
            \midrule
            $ 1$ & $2.414$ & $2.406$ & $2.336$ \\
            $3$ & $1.387$ & $\mathbf{1.387}$ & $1.370$ \\
            $5$ & $1.219$ & $\mathbf{1.219}$ & $1.212$ \\
            \bottomrule
        \end{tabular}
        \caption{$\rho = \rho_2, \rho_\infty$ in dimension $m=1$}
        \label{tab:Rungerates}
    \end{subtable}%
    \hfill
    \begin{subtable}[t]{0.45\textwidth}
        \centering
        \begin{tabular}{c|ccc}
            \toprule
            \multirow{2}{*}{$r$} &
            \multicolumn{3}{c}{$\rho$} \\
            \cline{2-4}
            &
            $m=2$ &
            $m=3$ &
            $m=4$ \\
            \midrule
            $ 1$ & $1.931$ & $1.732$ & $1.618$ \\
            $3$ & $1.263$ & $1.210$ & $1.180$ \\
            $5$ & $1.151$ & $1.122$ & $1.104$ \\
            \bottomrule
        \end{tabular}
        \caption{$\rho=\rho_1$ in dimensions $m=2,3,4$}
        \label{Rungerates1}
    \end{subtable}
    \caption{Dimension-independent optimal approximation rates for the Runge function \ref{F2} for different $r$, $p$, and $m$. The numerically achieved rates for dimension $m=1$ are reported in the left table, with bold indicating optimal achieved rates.}
\end{table}

\subsubsection{Discussion of results for \ref{F2}}
The dimension-independent optimal approximation rates for Euclidean and maximum degrees ($p =2,\infty$, Eq.~\eqref{eq:Rbounds}) are given in Table~\ref{tab:Rungerates}, along with the achieved rates for the LCL-node and LP-node interpolants in dimension $m=1$. The dimension-dependent optimal rates for the total-degree case ($p=1$) are reported in Table~\ref{Rungerates1}. The approximation errors for the interpolants of the Runge function \ref{F2} and their derivatives are plotted in Fig.~\ref{Rungerates}. The numerically reached approximation rates of all conducted cases are reported in Tables~\ref{tab:2} and \ref{tab:diff-runge}.

Both LCL- and LP-node interpolants approximate \ref{F2} with rates close to the optimum in the total- and Euclidean-degree cases. However, LCL-node interpolants perform than those on LP ones in the maximum-degree case ($p=\infty$), reaching optimal rates. While all interpolants reach rates close to optimal for the $1$st and $2$nd derivatives, LP-node interpolants seem to have a slight advantage here (Table~\ref{tab:diff-runge}).

\begin{table}[!htp]
    \centering
    \begin{tabular}{c|c|cc|cc|cc}
    \toprule
    \multirow{2}{*}{$r$} &
    \multirow{2}{*}{$p$} &
    \multicolumn{2}{c|}{$m = 2$} &
    \multicolumn{2}{c|}{$m = 3$} &
    \multicolumn{2}{c}{$m = 4$} \\
    \cline{3-8}
    &
    &
    LCL &
    LP &
    LCL &
    LP &
    LCL &
    LP \\
    \midrule

                 & $1.0$     & $1.911$ & $1.896$ & $1.700$ & $1.703$ & $1.585$ & $1.584$ \\
$ 1$ & $2.0$  &   $2.332$ & $2.351$ & $2.313$ & $2.353$ & $2.303$ & $2.360$ \\
                 & $\infty$ & $2.408$ & $2.349$ & $2.412$ & $2.359$ & $2.408$ & $2.371$ \\
\midrule
                 & $1.0$    & $1.252$ & $1.255$ & $1.201$ & $1.204$ & $1.175$ & $1.169$ \\
$3$  & $2.0$   & $1.360$ & $1.373$ & $1.370$ & $1.375$ & $1.345$ & $1.369$ \\
                 & $\infty$ & $\mathbf{1.387}$ & $1.372$ & $\mathbf{1.387}$ & $1.367$ & $1.396$ & $1.381$ \\
\midrule
                 & $1.0$    & $1.145$ & $1.147$ & $1.116$ & $1.115$ & $1.095$ & $1.101$ \\
$5$  & $2.0$     & $1.206$ & $1.212$ & $1.208$ & $1.209$ & $1.192$ & $1.211$ \\
                 & $\infty$  & $\mathbf{1.219}$ & $1.209$ & $\mathbf{1.219}$ & $1.209$ & $1.241$ & $1.208$ \\
 \bottomrule
 \end{tabular}
 \caption{Fitted approximation rates of LCL-node and LP-node interpolants of the Runge function \ref{F2} for different $r$, $p$, and $m$. Cases in which the optimal rates were achieved are marked bold.}
 \label{tab:2}
\end{table}

\begin{table}[!t]
    \centering
    \begin{tabular}{c|cc|cc}
    \toprule

    \multirow{2}{*}{$p$} &
    \multicolumn{2}{c}{$\frac{\partial f}{\partial x_1}$} &
    \multicolumn{2}{c}{$\frac{\partial^2 f}{\partial x_1^2}$} \\
    \cline{2-5}
    &
    LCL &
    LP &
    LCL &
    LP \\
    \midrule

$1.0$ &
$1.191$ &
$1.193$ &
$1.169$ &
$1.171$ \\

$2.0$ &
$1.338$ &
$1.364$ &
$1.310$ &
$1.325$ \\

$\infty$ &
$1.365$ &
$1.349$ &
$1.334$ &
$1.333$ \\
 \bottomrule
 \end{tabular}
 \caption{Fitted approximation rates of the $1$st and $2$nd derivatives of LCL-node and LP-node interpolants of the Runge function \ref{F2} for different $p$ in dimension $m = 3$, with $r = 3$.}
 \label{tab:diff-runge}
\end{table}

\begin{table}[!t]
    \centering
    \begin{tabular}{c|cc|cc|cc}
    \toprule

    \multirow{2}{*}{$p$} &
    \multicolumn{2}{c|}{$m = 2$} &
    \multicolumn{2}{c|}{$m = 3$} &
    \multicolumn{2}{c}{$m = 4$} \\
    \cline{2-7}

    &
    LCL &
    LP &
    LCL &
    LP &
    LCL &
    LP \\
    \midrule

  $1.0$     & $1.184$ & $1.183$ & $1.175$ & $1.178$ & $1.168$ & $1.159$ \\
  $2.0$     & $1.205$ & $1.207$ & $1.200$ & $1.202$ & $1.186$ & $1.209$ \\
  $\infty$   & $\mathbf{1.220}$ & $1.209$ & $\mathbf{1.220}$ & $1.209$ & $\mathbf{1.220}$ & $1.216$ \\

\bottomrule
 \end{tabular}
 \caption{Fitted approximation rates of LCL-node and LP-node interpolants of the function \ref{F3} for different $p$ and $m$. Cases in which the optimal rates were achieved are marked bold.\label{tab:F3}}
\end{table}

\subsubsection{Discussion of results for \ref{F3}} Adapting the computation of the optimal rate from \citet{bos2018bernstein}, the dimension-independent rate of \ref{F2} for $r=5$ also applies to \ref{F3} in the Euclidean- and maximum-degree cases. As reported in Table~\ref{tab:F3}, both LCL-node and LP-node interpolants achieve rates close to optimal in the Eucidean-degree case. The LCL-node interpolants reach the optimal rate in the maximum-degree case, while using LP nodes results in rates close to the optimum, reflecting the larger Lebesgue constant in this case (Lemma~\ref{lemma:LEB}). As expected, the rates in the total-degree case are lower and, as for \ref{F2}, decrease with increasing dimension $m$.

\subsubsection{Discussion of results for \ref{F4}} The numerically achieved rates for interpolating the function \ref{F4} are reported in Table~\ref{tab:F4}.
 In dimension $m=2$ and for $a = 5/4$, \cite{bos2018bernstein} computed the optimal rates $\rho_2 = 2.0518 <  2.1531 = \rho_\infty$ for the Euclidean- and maximum-degree cases. Our achieved rates for those cases are close to these predictions. As before, we observe significantly lower achieved rates in the total-degree case. In all cases, the difference between the Euclidean- and maximum-degree performance decreases with decreasing $a$. The rates only differ marginally between LCL- and LP-node interpolants.

\begin{table}[!htp]
    \centering
    \begin{tabular}{c|c|cc|cc|cc}
    \toprule

    \multirow{2}{*}{a} &
    \multirow{2}{*}{$p$} &
    \multicolumn{2}{c|}{$m = 2$} &
    \multicolumn{2}{c|}{$m = 3$} &
    \multicolumn{2}{c}{$m = 4$} \\
    \cline{3-8}

    &
    &
    LCL &
    LP &
    LCL &
    LP &
    LCL &
    LP \\
    \midrule

               & $1.0$     & $1.634$ & $1.634$ & $1.651$ & $1.654$ & $1.654$ & $1.667$ \\
$\frac{5}{4}$  & $2.0$     & $2.032$ & $1.993$ & $2.183$ & $2.190$ & $2.339$ & $2.324$ \\
               & $\infty$  & $2.148$ & $2.110$ & $2.295$ & $2.268$ & $2.423$ & $2.365$ \\
\midrule
                    & $1.0$      & $1.412$ & $1.411$ & $1.422$ & $1.421$ & $1.422$ & $1.421$ \\
$\frac{9}{8}$   & $2.0$      & $1.639$ & $1.626$ & $1.746$ & $1.734$ & $1.806$ & $1.806$ \\
                    & $\infty$   & $1.726$ & $1.695$ & $1.804$ & $1.773$ & $1.888$ & $1.869$ \\
\midrule
                    & $1.0$      & $1.275$ & $1.275$ & $1.286$ & $1.282$ & $1.283$ & $1.277$ \\
$\frac{17}{16}$ & $2.0$     & $1.421$ & $1.409$ & $1.481$ & $1.474$ & $1.510$ & $1.511$ \\
                    & $\infty$   & $1.473$ & $1.455$ & $1.522$ & $1.511$ & $1.595$ & $1.584$ \\
    \bottomrule
    \end{tabular}
    \caption{Fitted approximation rates of LCL-node and LP-node interpolants of the function \ref{F4} for different $a$, $p$, and $m$.\label{tab:F4}}
\end{table}

\subsubsection{Discussion of results for \ref{F5}} The achieved approximation rates for the interpolants of the function \ref{F5} and their $1$st and $2$nd derivatives are reported in Tables~\ref{tab:F5} and \ref{tab:diff-f5}. They clearly show that Euclidean-degree interpolation performs better than maximum-degree interpolation in this case. One might detect the increasing rate in the plots in Figs.~\ref{F5:a} and \ref{F5:b} for the Euclidean-degree case, reflecting the asymptotically unbounded $\rho$ for this function. Differences in the rates between LCL- or LP-node interpolants are mostly negligible, while LP nodes again appear to offer a slight advantage when computing derivatives of the interpolatns in the Euclidean-degree case (Table~\ref{tab:diff-f5}).

\begin{figure}[!htp]
\centering
    \captionsetup{width=1.0\linewidth}
    \vspace{-1.0em}

    \begin{subfigure}[b]{0.48\linewidth}
        \centering
        \includegraphics[scale=0.3]{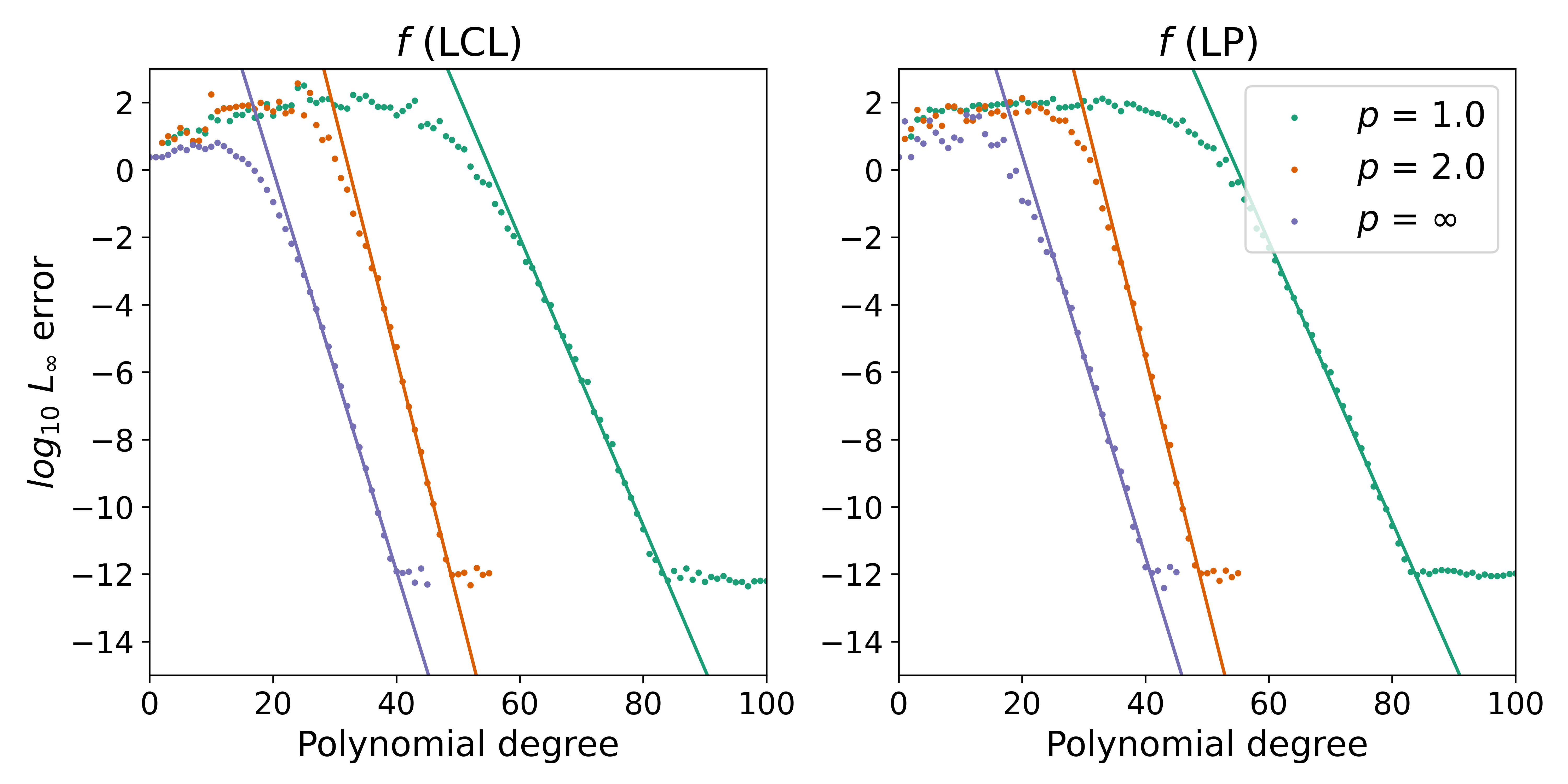}
        \caption{\ref{F5} in dimension $m =3$ with $k_1 = k_2 = 5$}    \label{F5:a}
    \end{subfigure}
    \hfill
        \begin{subfigure}[b]{0.48\linewidth}
        \centering
        \includegraphics[scale=0.3]{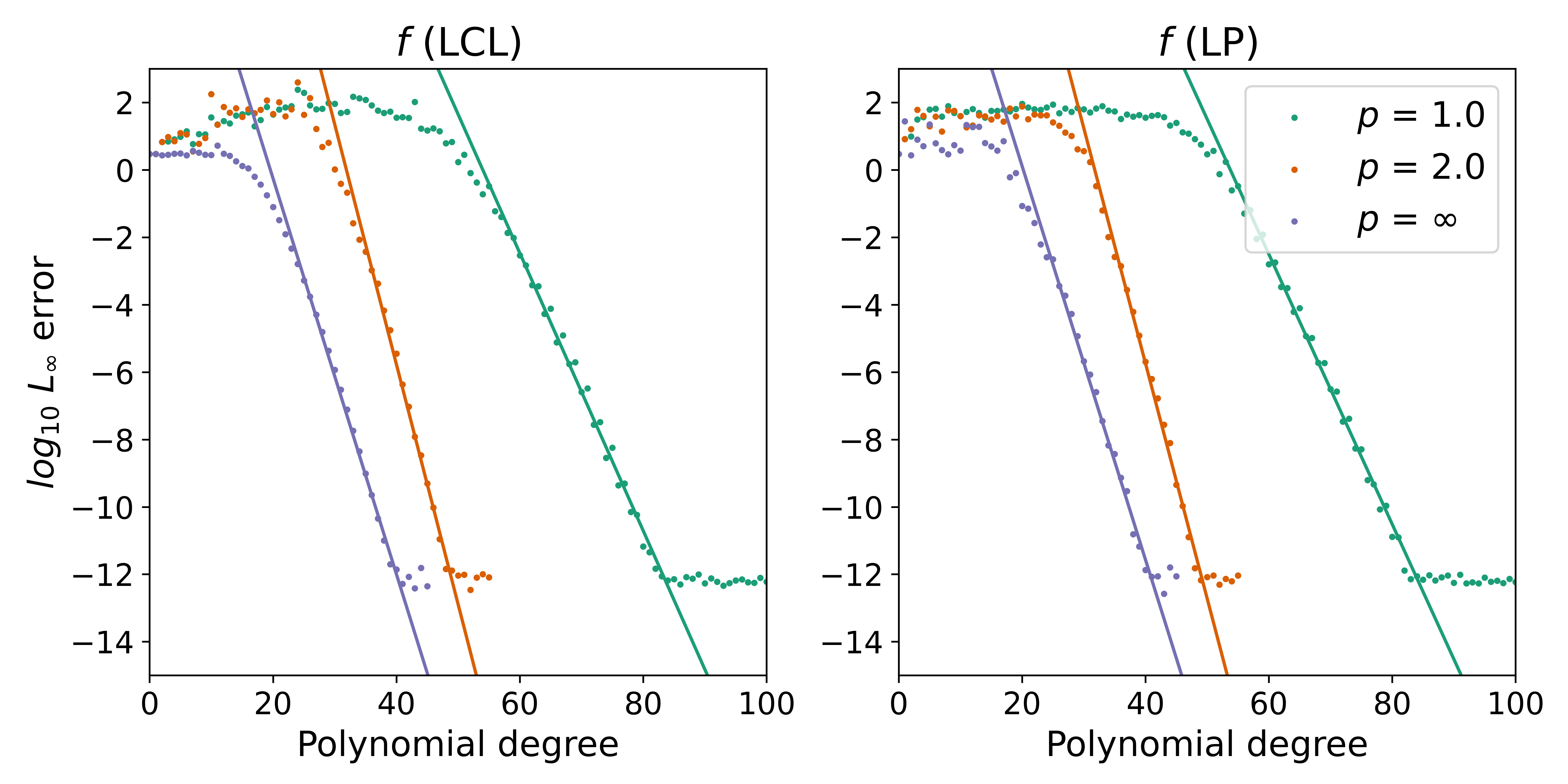}
        \caption{\ref{F5} in dimension $m =3$ with $k_1 = 1$, $k_2 = 5$}     \label{F5:b}
    \end{subfigure}

    \vspace{1em}
    \begin{subfigure}[b]{0.48\linewidth}
        \centering
        \includegraphics[scale=0.3]{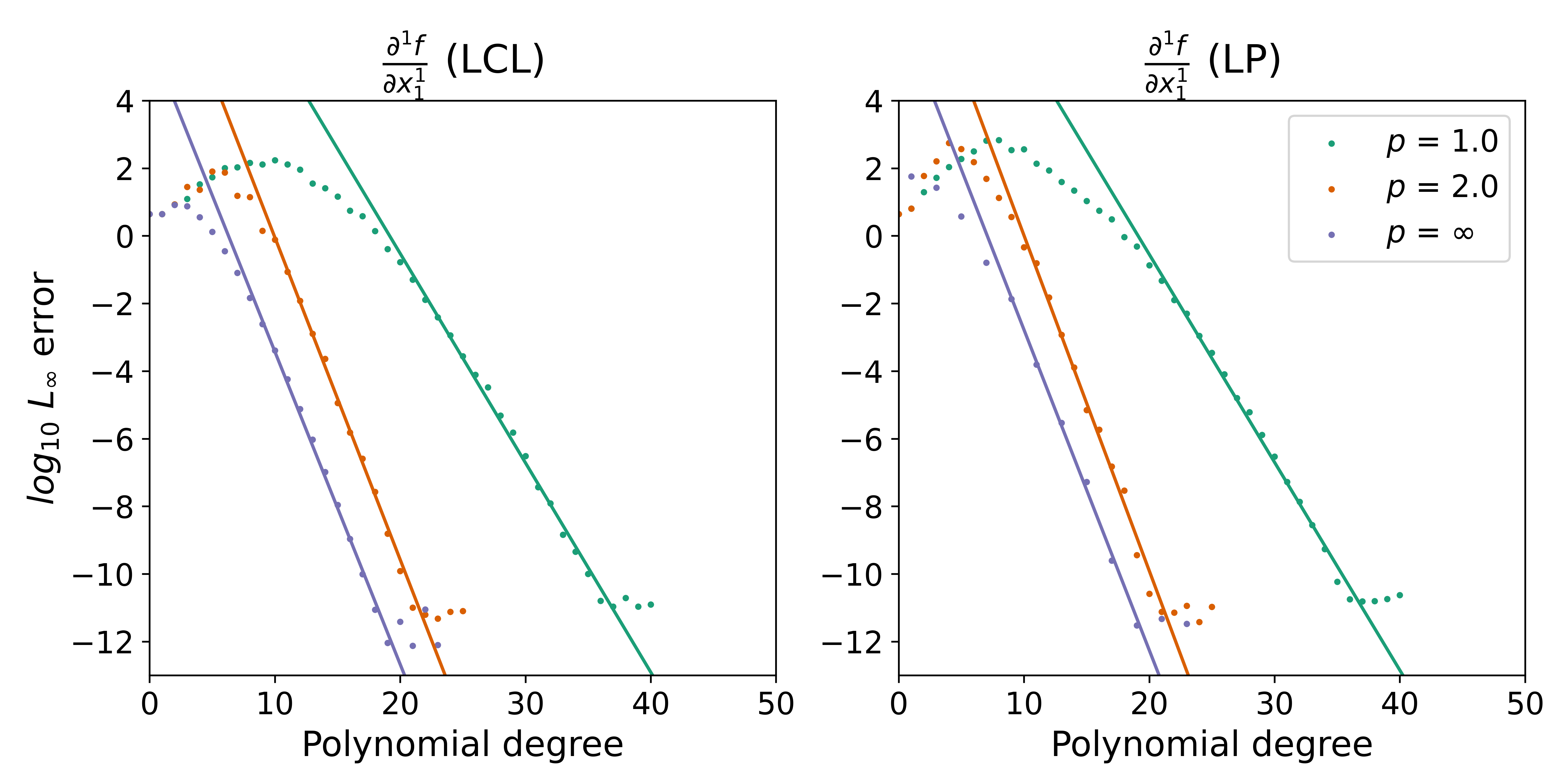}
        \caption{$1$st derivative
        of \ref{F5} in dimension $m =4$ with $k_1 = k_2 = 1$}
    \end{subfigure}%
    \hfill
    \begin{subfigure}[b]{0.48\linewidth}
        \centering
        \includegraphics[scale=0.3]{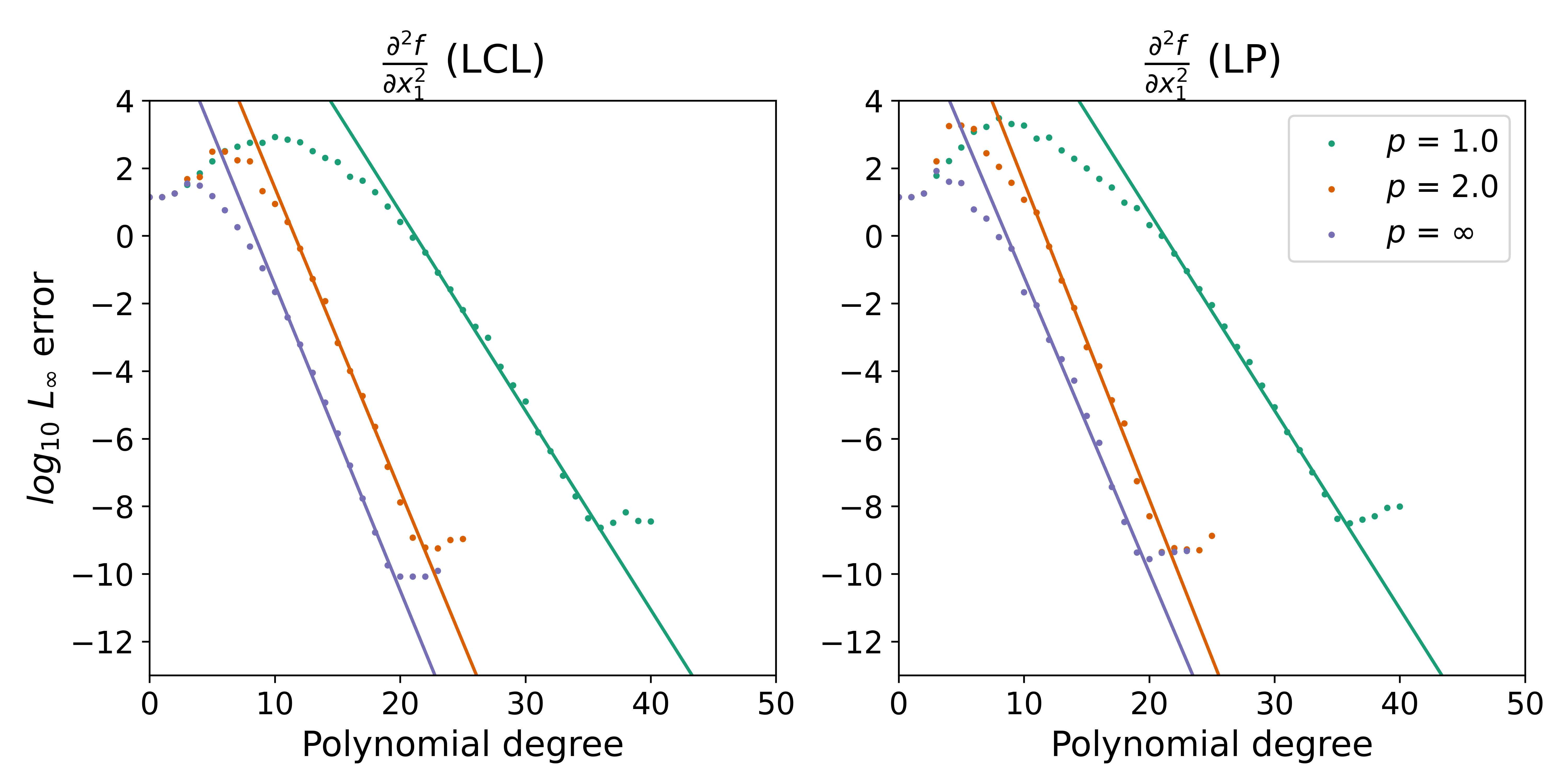}
        \caption{$2$nd derivative \ref{F5} in dimension $m =2$ with $k_1 = k_2 = 1$}
    \end{subfigure}

    \caption{Approximation errors of the LCL- and LP-node interpolants of \ref{F5} and their $1$st and $2$nd derivatives for $p=1,2,\infty$. Solid lines show the fitted rates from Tables~\ref{tab:F5} and\ref{tab:diff-f5}.}
    \label{ratesF5}
\end{figure}

\begin{table}[!t]
    \centering
    \begin{tabular}{c|c|cc|cc|cc}
    \toprule

    \multirow{2}{*}{$k_1$,$k_2$} &
    \multirow{2}{*}{$p$} &
    \multicolumn{2}{c|}{$m = 2$} &
    \multicolumn{2}{c|}{$m = 3$} &
    \multicolumn{2}{c}{$m = 4$} \\
    \cline{3-8}

    &
    &
    LCL &
    LP &
    LCL &
    LP &
    LCL &
    LP \\
    \midrule

         & $1.0$    & $6.275$ & $6.432$ & $5.271$ & $5.226$  & $4.377$  & $4.675$ \\
$1$, $1$ & $2.0$    & $9.595$ & $9.511$ & $9.849$ & $10.765$ & $11.643$ & $11.359$ \\
         & $\infty$ & $9.603$ & $9.480$ & $9.325$ & $9.139$  & $8.950$  & $8.590$ \\
\midrule
                                         & $1.0$ & $3.879$ & $4.029$ & $3.275$ & $3.120$  & $2.878$ & $2.739$ \\
$3$, $3$ & $2.0$    & $6.357$ & $6.528$ & $7.239$ & $7.020$  & $8.234$ & $8.162$ \\
                                         & $\infty$  & $5.411$ & $5.482$ & $5.034$ & $5.006$  & $5.233$ & $5.102$ \\
\midrule
                                         & $1.0$    & $3.105$ & $3.036$ & $2.673$ & $2.605$  & $2.461$ & $2.486$ \\
$5$, $5$ & $2.0$     & $4.611$ & $4.650$ & $5.342$ & $5.400$  & $5.741$ & $5.668$ \\
                                         & $\infty$  & $3.929$ & $3.865$ & $3.930$ & $3.947$  & $4.017$ & $3.997$ \\
\midrule
         & $1.0$    & $2.922$ & $2.759$ & $2.581$ & $2.516$  & $2.468$ & $2.494$ \\
$1$, $5$ & $2.0$    & $4.596$ & $4.425$ & $5.151$ & $4.987$  & $5.889$ & $5.671$ \\
         & $\infty$ & $3.958$ & $4.027$ & $3.865$ & $3.837$  & $4.024$ & $3.984$ \\
 \bottomrule
 \end{tabular}
 \caption{Fitted approximation rates of LCL-node and LP-node interpolants of the function \ref{F5} for different $k_1$, $k_2$, $p$, and dimensions $m$.\label{tab:F5}}
\end{table}

\begin{table}[!t]
    \centering
    \begin{tabular}{c|cc|cc}
    \toprule
    \multirow{2}{*}{$p$} &
    \multicolumn{2}{c}{$\frac{\partial f}{\partial x_1}$} &
    \multicolumn{2}{c}{$\frac{\partial^2 f}{\partial x_1^2}$} \\
    \cline{2-5}
    &
    LCL &
    LP &
    LCL &
    LP \\
    \midrule


$1.0$ &
$4.161$ &
$4.120$ &
$3.876$ &
$3.859$ \\

$2.0$ &
$8.968$ &
$9.816$ &
$7.866$ &
$8.671$ \\

$\infty$ &
$8.406$ &
$8.866$ &
$8.015$ &
$7.496$ \\
 \bottomrule
 \end{tabular}
 \caption{Fitted approximation rates of the $1$st and $2$nd derivatives of LCL-node and LP-node interpolants of the function \ref{F5} for different $p$ in dimension $m = 4$ for $k_1 = k_2 = 1$.}
 \label{tab:diff-f5}
\end{table}

\section{Conclusion}\label{sec:Conclusion}
We have proven and numerically demonstrated that Bos--Levenberg--Trefethen functions can be optimally approximated by multivariate Newton interpolation in downward-closed spaces in non-tensorial nodes, achieving geometric rates. In particular, the Euclidean-degree case mitigates the curse of dimensionality for interpolation tasks. By maintaining both efficiency and approximation power, even for the derivatives of the interpolants, this might establish a new standard in spectral methods for regular partial differential equations (PDEs), ordinary differential equations (ODEs), and signal-processing problems.

We also presented an algorithm to practically compute the described interpolants in quadratic time $\Oc(N^2)$ and linear storage $\Oc(N)$, where $N = \dim \Pi_A$ is the dimension of the  downward-closed space. We provided an open-source implementation of the algorithm in Python as the {\sc minterpy} package, and we validated and benchmarked it in numerical experiments.

We believe, however, that the present algorithm is not yet optimal and see potential for further reducing its time complexity to $\Oc(N m n)$. This would render it even faster than tensorial \emph{Fast Fourier Transform}, which has a complexity of $\Oc(M\log(M))$, $M = (n+1)^m \gg N$. Then, $N \ll M$ holds for a degree range $1 \leq n \leq R$, where $R>1$ is growing with dimension $m$. Realizing such a \emph{Fast Newton Transform} is the focus of our future work.


\section*{Acknowledgements}
We deeply acknowledge Leslie Greengard, Albert Cohen, Christian L.~M\"{u}ller, Alex Barnett, Manas Rachh, Heide Meissner, Uwe Hernandez Acosta, and Nico Hoffmann for many inspiring comments and helpful discussions.
We are grateful to Michael Bussmann and CASUS (G\"{o}rlitz, Germany) for hosting stimulating workshops on the subject.

This work was partially funded by the Center of Advanced Systems Understanding (CASUS), financed by Germany's Federal Ministry of Education and Research (BMBF) and by the Saxon Ministry for Science, Culture and Tourism (SMWK) with tax funds on the basis of the budget approved by the Saxon State Parliament.

\bibliographystyle{IMANUM-BIB}
\bibliography{Ref.bib}
\clearpage

\appendix

\end{document}